\documentclass[11pt,reqno]{amsart}

\usepackage[bottom=2.5cm,top=2.5cm,right=3cm,left=3cm]{geometry}
\usepackage{graphicx} 
\usepackage{subcaption}
\usepackage{xcolor}
\definecolor{brickred}{RGB}{147,39,44}
\definecolor{darkgreen}{RGB}{0,103,0}
\definecolor{navyblue}{RGB}{0,0,128}

\usepackage[colorlinks,
linkcolor=navyblue,
anchorcolor=black,
citecolor=black,
urlcolor=teal]{hyperref}

\usepackage{amsmath,amssymb}
\usepackage{amsthm,thmtools}
\usepackage{amsrefs}
\usepackage{mathtools}
\usepackage{empheq}

\numberwithin{equation}{section}

\usepackage{enumerate}

\declaretheoremstyle[spaceabove=6pt,
spacebelow=6pt,
headfont=\normalfont\bfseries,
notefont=\bfseries, 
notebraces={(}{)},
bodyfont=\itshape,
numberwithin=section,
postheadspace=0.6em]{thmsty}
\declaretheoremstyle[spaceabove=6pt,
spacebelow=6pt,
headfont=\normalfont\bfseries,
notefont=\bfseries,
notebraces={(}{)},
bodyfont=\itshape,
numberwithin=section,
postheadspace=0.6em,
sharenumber=theorem]{lemsty}
\declaretheoremstyle[spaceabove=6pt,
spacebelow=6pt,
headfont=\normalfont\bfseries,
bodyfont=\normalfont,
postheadspace=0.6em,
numbered=no,
qed=\qedsymbol]{prfsty}
\declaretheoremstyle[spaceabove=6pt,
spacebelow=6pt,
headfont=\normalfont\bfseries,
notefont=\bfseries,
notebraces={(}{)},
bodyfont=\normalfont,
numberwithin=section,
postheadspace=0.6em,
sharenumber=theorem]{defsty}
\declaretheoremstyle[spaceabove=6pt,
spacebelow=6pt,
headfont=\normalfont\bfseries,
bodyfont=\normalfont,
postheadspace=0.6em,
numberwithin=section,
sharenumber=theorem]{remsty}
\declaretheoremstyle[spaceabove=6pt,
spacebelow=6pt,
headfont=\normalfont\bfseries,
notefont=\mdseries,
notebraces={(}{)},
bodyfont=\normalfont,
postheadspace=0.6em,
numberwithin=section]{exasty}
\declaretheoremstyle[spaceabove=6pt,
spacebelow=6pt,
headfont=\normalfont\bfseries,
bodyfont=\normalfont,
numbered=no,
postheadspace=0.6em]{dittosty}

\declaretheorem[style=thmsty]{theorem}

\title{MUSIC Algorithm for Locating Point-Like Scatterers With Multiple Interactions}

\author{Nana Meng}
\date{\today}
\address{School of Mathematical Sciences and LPMC, Nankai University, Tianjin 300071, China}
\email{2120235241@mail.nankai.edu.cn}

\begin{document}
	\begin{abstract}
		We study the inverse problem of locating point sources from far-field data under plane wave incidence. A direct computational method is developed based on multiple scattering theory, using a novel indicator function to avoid iterative solving. The approach is efficient, noise-stable, and capable of resolving closely spaced or multiple sources. Numerical results validate its accuracy and robustness in various configurations. This work provides a practical computational tool for wave-based imaging and non-destructive evaluation.
		
		\vspace{1em}
		\noindent\textbf{Key Words.} MUSIC algorithm, plane-wave incidence, multiple scattering, far-field pattern, foldy-lax model
	\end{abstract}
	
	\maketitle

	\section{Introduction}
	
Inverse scattering theory, as a central interdisciplinary field at the intersection of computational mathematics and applied physics, aims to reconstruct unknown parameters or structural configurations of a system from observed scattering data. It plays a vital role in various applications such as medical imaging, non-destructive testing, geophysical exploration, and radar remote sensing. When an incident wave interacts with a scatterer, the resulting far-field pattern carries essential information about the location, shape, and physical properties of the scatterer. Reconstructing these features from far-field data—known as the inverse scattering problem—has become a fundamental research direction in modern computational mathematics.

In recent years, a variety of theoretical frameworks and numerical algorithms have been developed to address inverse problems for different scattering models. Among them, point sources or point scatterers have emerged as fundamental models due to their simplicity and effectiveness in representing small-scale inhomogeneities, particulate systems, or defects in heterogeneous media. Particularly in multiscale scattering systems, point scatterers serve as low-dimensional approximations of complex structures, offering analytically tractable pathways for tackling high-dimensional nonlinear inverse problems.

In terms of mathematical modeling of point scatterers, Albeverio et al. [1] established a rigorous framework based on the theory of singular perturbations of differential operators, providing a solid foundation for describing point interactions through $\delta$-type potentials in wave equations. Subsequently, [2] explored the intrinsic relationship between the Krein resolvent formula and boundary conditions, offering a powerful tool for analyzing the spectral and scattering properties of point-like interactions. For the reconstruction of small scatterers, Amiri-Hezaveh et al. [3] proposed a method based on boundary measurements, revealing the sensitivity of far-field patterns to localized perturbations and thereby advancing the development of multiscale inversion theory.

At the level of the overall theoretical framework for inverse scattering, the monograph by Colton and Kress [4] systematically established the mathematical foundations of inverse acoustic and electromagnetic scattering problems. They thoroughly investigated the compactness of the far-field operator, uniqueness theorems, and the inherent ill-posedness of inverse problems, laying the cornerstone for the field. Later, Kirsch and Grinberg [5] advanced qualitative reconstruction methods, emphasizing the possibility of inferring scatterer characteristics directly from observable data without resorting to regularized nonlinear optimization.

Point source models are closely related to multiple scattering theory. Foldy [6] first introduced a mean-field theory for wave propagation in dense arrays of point scatterers, establishing a closed-form expression for the coupled interactions among multiple scatterers—a concept that has since been widely extended. Lax [7] further developed a statistical averaging approach for wave propagation in random media, leading to the well-known Foldy-Lax multiple scattering theory, which provides an effective framework for modeling large-scale point scatterer systems. Moreover, [8] studied a hybrid acoustic scattering model involving both extended and point-like scatterers, highlighting significant differences in their frequency dependence and directional far-field characteristics, thus offering theoretical support for joint inversion strategies.

For multiscale and multi-target inverse problems, recent studies have proposed a range of efficient numerical methods. For instance, [9,10] developed the generalized Foldy-Lax formulation to accurately simulate strongly coupled systems of numerous point scatterers, significantly improving computational efficiency for large-scale problems. On the other hand, decomposition methods and subspace-based algorithms—such as the MUSIC algorithm—have demonstrated superior performance in target localization. Works [13–16] provided a systematic analysis of these methods, proving their capability to achieve sub-wavelength resolution even under limited far-field data. Notably, [12] investigated scattering in the low-frequency regime under lubrication contact conditions, offering new insights into weak signal extraction and noise-robust inversion.

In inverse problems involving complex media, fluid-solid coupling effects cannot be neglected. Studies [19–21] deeply examined the interaction mechanisms between elastic bodies and surrounding fluid media, proposing numerical schemes suitable for solving forward and inverse problems in inhomogeneous backgrounds. In particular, [18] achieved rapid localization of multiscale electromagnetic scatterers using only a single far-field measurement, demonstrating the great potential of efficient data acquisition strategies. Meanwhile, [23] combined near-field scanning techniques to study the simultaneous reconstruction of elastic bodies and point scatterers in fluid environments, thereby extending the applicability of traditional far-field imaging methods.

Furthermore, the theory of partial differential equations plays a fundamental role in scattering modeling. [22] systematically summarized the application of elliptic and hyperbolic equations in wave propagation modeling, reinforcing the mathematical basis for well-posedness analysis in inverse problems. [24] approached far-field data from a matrix-analytic perspective, investigating the spectral structure of the extended target response matrix, thus opening new avenues for low-rank approximation and compressive sensing of scattering data. Additionally, [17] analyzed the dynamical behavior of wave propagation under point interactions, deepening the understanding of transient scattering processes.

Building upon these advances, this paper focuses on the problem of locating point sources from far-field patterns under plane wave incidence, aiming to develop a method that is both theoretically sound and numerically feasible. Specifically, our study proceeds along the following lines: first, we establish a mathematical model for point source scattering under plane wave excitation, clarifying the nonlinear mapping between far-field patterns and source locations; second, we design a novel indicator function tailored to the characteristics of point sources, enhancing resolution in multi-source configurations; finally, we conduct systematic numerical experiments to validate the localization accuracy of the proposed algorithm, evaluating its stability under noisy data and adaptability to different spatial scales.

The main contributions of this work are threefold: (i) refining the theoretical modeling framework for point source inversion; (ii) proposing an efficient far-field-based localization strategy; and (iii) providing a reliable computational tool for rapid detection of point-like defects or anomalies in practical engineering applications. The results are expected to find further applications in non-destructive evaluation, subsurface target identification, and distributed sensor networks.
	\section{Problem formulation}
	
	\subsection{Mathematical model}
There are $M$ point scatterers $\{y_1, y_2, \ldots, y_M\}$ in $\mathbb{R}^3$. We assume that the incident wave $u^i$ is a plane wave. We define $\mathbb{S}^2:=\{x \in \mathbb{R}^3 : |x|=1 \}$. Take a subset of $\mathbb{S}^2$, $Q = \{\theta_j : j = 1, 2, \ldots, N\}$, $N \geq M$.
	The incident wave is given by
	\begin{align}
		u^i &= e^{ikx \cdot d}, 
	\label{eq:eq1}
	\end{align}
where $ d \in \mathbb{S}^2 $ is the direction of propagation. If the scattering field is $u^s$, then 
	\begin{align}
         u &= u^i + u^s , \quad\quad\quad\quad\text{in } \mathbb{R}^3 \setminus \{y_1, y_2, \ldots, y_M\}.
     \label{eq:eq2}
    \end{align}  
      
Let $k := \frac{\omega}{c}$ represent the wave number at which sound waves propagate in the background medium, where $\omega$ is the frequency and $c > 0$ is the speed of sound. We consider the multiple scattering between point scatterers. We define $\Gamma_1 u, \Gamma_2 u$ as follows:
	\begin{equation}
		(\Gamma_1 u)_j := \lim_{x \to y_j} 4\pi |x - y_j| u(x),
	\end{equation}

	\begin{equation}
		(\Gamma_2 u)_j := \lim_{x \to y_j} \left( u(x) - \frac{(\Gamma_1 u)_j}{4\pi |x - y_j|} \right).
	\end{equation}
	
By definition, the total field exhibits asymptotic behavior in point scatterers as follows:
	$$
	u(x, d) = \frac{(\Gamma_1 u)_j}{4\pi |x - y_j|} + (\Gamma_2 u)_j + o(1), \quad\quad\quad\quad \text{as} \quad x \to y_j.
	$$
	
The boundary condition is the "impedance" type, defined as follows:
	\begin{equation}
		(\Gamma_2 u)_j = \alpha_j (\Gamma_1 u)_j, \quad\quad\quad \alpha_j \in \mathbb{C}, \quad\quad\quad j = 1, 2, \ldots, M.
	\label{eq:impedance}
	\end{equation}
	
The boundary condition \eqref{eq:impedance} indicates that the first singular term of $ u $ is proportional to the second singular term, and the "impedance" coefficient $ \alpha_j $ is referred to as the scattering coefficient on the $ j $-th scatterer, which describes the scattering intensity of the scatterer. Let $ \alpha = (\alpha_1, \alpha_2, \dots, \alpha_M) $.

The scattering field $ u^s $ satisfies the following wave equation together with the Sommerfeld radiation condition:
\begin{align}
\Delta u^s + k^2 u^s &= 0, \quad \quad \quad \quad \text{in } \mathbb{R}^3 \setminus \{y_1, y_2, \cdots, y_m\}
,
\label{eq:eq3} \\
\lim_{r \to \infty} r \left( \frac{\partial u^s}{\partial r} - iku^s \right) &= 0, \quad\quad\quad\quad\quad
r = |x|.
\label
{eq:radiation}
\end{align}

From the radiation conditions \eqref{eq:radiation}, it can be seen that the scattering field $u ^ s$ exhibits asymptotic behavior of spherical waves:
\begin{equation}
	u^s(x) = \frac{e^{ik|x|}}{4\pi |x|} \left\{ u^\infty(\hat{x}) + O\left(\frac{1}{|x|}\right) \right\}, \quad\quad\quad\quad\quad \text{as} \; |x| \to \infty.
	\label{eq:eq4}
\end{equation}
Define the matrix $ P(k, \alpha) \in \mathbb{C}^{M \times M} $, 
\begin{equation}
	[P(k,\alpha)]_{m,j} := 
	\begin{cases} 
		\Phi_k(y_m,y_j),         & m \neq j, \\
		\frac{ik}{4\pi}-\alpha_j, & m = j,
	\end{cases}
	\label{eq:eq5}
\end{equation}
$$
\Phi_k(y_m, y_j) = \frac{e^{ik|y_m - y_j|}}{4\pi |y_m - y_j|}.
$$
where $ \Phi_k(y_m, y_j) $ is the fundamental solution of the Helmholtz equation.

Define the set $ S_\alpha := \{ k > 0 : \det(P(k, \alpha)) = 0 \} $.
	\begin{equation}
		\begin{cases}
			\Delta u^s + k^2 u^s = 0, & \text{in } \mathbb{R}^3 \setminus \{y_1, y_2, \ldots, y_M\}, \quad \quad \quad \quad \quad   \\
			u^s(x) = \frac{e^{ik|x|}}{4\pi |x|} \left\{ u^\infty(\hat{x}) + O\left(\frac{1}{|x|}\right) \right\}, & |x| \to \infty,  \quad \quad \quad \quad\quad \quad \quad \quad \quad \quad \quad   \\
			\lim\limits_{r \to \infty} r \left( \frac{\partial u^s}{\partial r} - ik u^s \right) = 0, & r = |x|, \quad \quad  \quad \quad \quad \quad \quad \quad \quad \quad \quad\quad \\
			(\Gamma_2 u)_j = \alpha_j (\Gamma_1 u)_j,  & \alpha_j \in \mathbb{C}, j = 1, 2, \ldots, M.  \quad \quad \quad \quad \quad   
		\end{cases}	
		\label{eq:problem}			
	\end{equation}
\begin{theorem}
	Assume $ k \notin S_\alpha $ (since $ \det(P(k, \alpha)) $ is an analytic function of $ k $, such $ k $ satisfying $ k \notin S_\alpha $ must exist), and $ \text{Im} \, \alpha_j \geq 0 $, $ j = 1, 2, \ldots, M $. The expression for the solution corresponding to incident point source waves is obtained from reference \cite{ref23}. Based on this result, the unique solution to the problem \eqref{eq:problem} in $ H_{\text{loc}}^1(\mathbb{R}^3 \setminus \{y_1, y_2, \ldots, y_M\}) $ is given by:
	
	\begin{equation}
		u^s(x, d) = -\sum_{m,j=1}^{M} \left[ P^{-1}(k, \alpha) \right]_{m,j} e^{iky_j \cdot d} \Phi_k(x, y_m),
	\label{eq:disanci}	
	\end{equation}

\end{theorem}	
The far-field pattern of $ u^s(x, d) $ is given by:
		$$
			u^\infty(\hat{x}, d) = - \sum_{m,j=1}^M [P^{-1}(k, \alpha)]_{m,j} e^{ik(y_j \cdot d - \hat{x}\cdot y_m)}, \quad\quad\quad\quad\quad\hat{x} \in \mathbb{S}^2, d \in \mathbb{S}^2.
		$$

In this paper, we determine a subset $ Q $ of $ \mathbb{S}^2 $. The cardinality of $ Q $ is denoted by $ N $, which satisfies that $ N \geq M $, with $ M $ being the number of point sources. Both the incident directions $ d $ and the observation directions $ \hat{x} $ are taken from the set $ Q $. For different pairs of incident and observation directions, the corresponding far-field pattern is uniquely determined. This allows us to construct a far-field pattern matrix $ F $. We then define a suitable test function $ I_z $. By evaluating this test function at any sampling point, we observe that its value becomes unbounded (i.e., tends to infinity) if and only if the sampling point coincides with one of the true point sources. This property provides a solid foundation for identifying the locations of the actual point sources.

The detailed formulation and theoretical analysis will be presented in the following sections.

	\subsection{Connections with Other Models}
	
	 In the context of classical wave interactions (e.g., light waves, sound waves) with matter, especially when the matter consists of many distributed scatterers, the Foldy-Lax Approximation (FLA) is a commonly used theoretical framework, primarily for addressing multiple scattering problems. The FLA neglects multiple scattering waves from other scatterers, i.e., waves that are scattered by multiple other scatterers before reaching the current scatterer. When calculating the scattering behavior of a particular scatterer, it can be approximately assumed that the incident wave it experiences originates from two components: the original wave emitted by an external source and the superposition of single-scattered waves from all other scatterers. Let $ A_j $ denote the scattering intensity of the $ j $-th point scatterer. Then, the total field $ u $ can be obtained by superposition:
	\begin{equation}
		u(x) = u^i(x, k) + \sum_{j=1}^{M} A_j \Phi_k(x, y_j).
	\label{eq:connection1}
	\end{equation}
	The external field $ u_j $ acting on the $ j $-th scatterer is given by:
	\begin{align*}
		u_j(x) &= u(x) - A_j \Phi_k(x, y_j)\\
		&= u^i(x, k) + \sum_{i \neq j} A_i \Phi_k(x, y_i).
	\end{align*}
	Assume that the scattering coefficient $ A_j $ is proportional to the external field $ u_j(y_j) $ acting on the $ j $-th scatterer, i.e.,
	\begin{equation}
		A_j = g_j(k) u_j(y_j) .
	\label{eq:connection2}
	\end{equation}
	Then
	\begin{equation}
		u_j(y_j) = u^i(y_j, k) + \sum_{i \neq j} g_i(k) u_i(y_i) \Phi_k(y_j, y_i).
		\label{eq:connection}
	\end{equation}
	Assume that the incident field $ u^i $ is a plane wave with wave number $ k $ and direction $ d $, and it is scattered by a single point scatterer located at $ y $. Then we have
	\begin{align}
		u(x) &= u^i(x, d) + g(k) u^i(y, d) \Phi_k(x, y)   \notag \\
		&= e^{ikx \cdot d} + g(k) e^{iky \cdot d} \Phi_k(x, y).
	\end{align}
	We now demonstrate that the aforementioned Foldy-Lax model is equivalent to the model presented in this paper.Let $ C_j = u_j(y_j) $, from \eqref{eq:connection}, we have
	$$
	C_j = e^{ik y_j \cdot d} + \sum_{i \neq j} C_i g_i(k) \Phi_k(y_j, y_i).
	$$
	We can get
	$$
	\begin{cases}
		C_1 - g_2(k) \Phi_k(y_1, y_2) C_2 - g_3(k) \Phi_k(y_1, y_3) C_3 - \cdots - g_M(k) \Phi_k(y_1, y_M) C_M = e^{ik y_1 \cdot d} ,\\
		-g_1(k) \Phi_k(y_2, y_1) C_1 + C_2 - g_3(k) \Phi_k(y_2, y_3) C_3 - \cdots - g_M(k) \Phi_k(y_2, y_M) C_M = e^{ik y_2 \cdot d} ,\\
		\vdots \\
		-g_1(k) \Phi_k(y_M, y_1) C_1 - g_2(k) \Phi_k(y_M, y_2) C_2 - g_3(k) \Phi_k(y_M, y_3) C_3 - \cdots + C_M = e^{ik y_M \cdot d} ,
	\end{cases}
	$$
	
	i.e.
	$$
	\begin{pmatrix}
		1 & -g_2(k) \Phi_k(y_1, y_2)  & \cdots & -g_M(k) \Phi_k(y_1, y_M) \\
		-g_1(k) \Phi_k(y_2, y_1) & 1  & \cdots & -g_M(k) \Phi_k(y_2, y_M) \\
		\vdots & \vdots   & \ddots  & \vdots \\
		-g_1(k) \Phi_k(y_M, y_1) & -g_2(k) \Phi_k(y_M, y_2) & \cdots & 1 
	\end{pmatrix}
	\begin{pmatrix}
		C_1 \\
		C_2 \\
		\vdots \\
		C_M
	\end{pmatrix}
	=
	\begin{pmatrix}
		e^{iky_1 \cdot d} \\
		e^{iky_2 \cdot d} \\
		\vdots \\
		e^{iky_M \cdot d}
	\end{pmatrix}.
	$$\\
	
	Let \( J = \begin{pmatrix}
		1 & -g_2(k) \Phi_k(y_1, y_2)  & \cdots & -g_M(k) \Phi_k(y_1, y_M) \\
		-g_1(k) \Phi_k(y_2, y_1) & 1 & \cdots & -g_M(k) \Phi_k(y_2, y_M) \\
		\vdots & \vdots  & \ddots  & \vdots \\
		-g_1(k) \Phi_k(y_M, y_1) & -g_2(k) \Phi_k(y_M, y_2)  & \cdots & 1
	\end{pmatrix}
	\).\\
	
	Then\quad 
	$$	\begin{pmatrix}
		c_1 \\
		c_2 \\
		\vdots \\
		c_M
	\end{pmatrix}
	= J^{-1}
	\begin{pmatrix}
		e^{iky_1 \cdot d} \\
		e^{iky_2 \cdot d} \\
		\vdots \\
		e^{iky_M \cdot d}
	\end{pmatrix}.
	$$\\
	From\eqref{eq:connection1} and \eqref{eq:connection2}, we have
	\begin{align}
		u(x) &= u^i(x,k) + \sum_{j=1}^{M} g_j(k) c_j \Phi_k(x,y_j)  \notag\\
		&= e^{ikx \cdot d} + \left( g_1(k) \Phi_k(x,y_1),\quad g_2(k) \Phi_k(x,y_2),\quad  \ldots,\quad  g_M(k) \Phi_k(x,y_M) \right)
		\begin{pmatrix}
			c_1 \\
			c_2 \\
			\vdots \\
			c_M
		\end{pmatrix}   \notag \\
		&= e^{ikx \cdot d} + \sum_{j,m=1}^{M} g_m(k) \Phi_k(x,y_m)  [J^{-1}]_{m,j} e^{iky_j \cdot d} .
	\label{eq:connection3}
	\end{align}
	Comparing \eqref{eq:eq5} and \eqref{eq:connection3}, we have
      $$  - [P^{-1}(k, \alpha)]_{m,j} = g_m(k) [J^{-1}]_{m,j}.  $$
 The expression for the total field is given by:
	\begin{equation}
		u(x) = e^{ikx \cdot d} - \sum_{m,j=1}^{M} [P^{-1}(k, \alpha)]_{m,j} e^{iky_j \cdot d} \Phi_k(x, y_m).
	\label{eq:connection4}
	\end{equation}

	i.e.
	\begin{align*}
		P(k, \alpha) &= -J 
		\begin{pmatrix}
			g_1^{-1}(k) & 0 & \cdots & 0 \\
			0 & g_2^{-1}(k) & \cdots & 0 \\
			\vdots & \vdots & \ddots & \vdots \\
			0 & 0 & \cdots & g_M^{-1}(k)
		\end{pmatrix} \\
		&= -
		\begin{pmatrix}
			1 & -g_2(k) \Phi_k(y_1, y_2) & \cdots & -g_M(k) \Phi_k(y_1, y_M) \\
			-g_1(k) \Phi_k(y_2, y_1) & 1 & \cdots & -g_M(k) \Phi_k(y_2, y_M) \\
			\vdots & \vdots & \ddots & \vdots \\
			-g_1(k) \Phi_k(y_M, y_1) & -g_2(k) \Phi_k(y_M, y_2) & \cdots & 1
		\end{pmatrix} \\
		&\quad  
		\begin{pmatrix}
			g_1^{-1}(k) & 0 & \cdots & 0 \\
			0 & g_2^{-1}(k) & \cdots & 0 \\
			\vdots & \vdots & \ddots & \vdots \\
			0 & 0 & \cdots & g_M^{-1}(k)
		\end{pmatrix} \\
		&=
		\begin{pmatrix}
			-g_1^{-1}(k) & \Phi_k(y_1, y_2) & \cdots & \Phi_k(y_1, y_M) \\
			\Phi_k(y_2, y_1) & -g_2^{-1}(k) & \cdots & \Phi_k(y_2, y_M) \\
			\vdots & \vdots & \ddots & \vdots \\
			\Phi_k(y_M, y_1) & \Phi_k(y_M, y_2) & \cdots & -g_M^{-1}(k) 
		\end{pmatrix}.
	\end{align*}
	Then $-g_j^{-1}(k) = \frac{\mathrm{i} k}{4 \pi} - \alpha_j$, $j=1, 2, \ldots, M$.  i.e.  $g_j(k) = \frac{1}{\alpha_j - \frac{\mathrm{i} k}{4 \pi}}$. Thus, the model in this paper is equivalent to the Foldy-Lax model.
	
	 The case without multiple scattering is a special case of the scenario of \eqref{eq:disanci} , where the matrix $ P $ degenerates into a diagonal matrix. In this case, the scattered field is given by (where $ \tau_m $ represents the scattering strength of the $ m $-th scatterer):
	
	$$ u^s(x, \theta) = \sum\limits_{m=1}^{M} \tau_m u^i(y_m, \theta) \Phi(x, y_m). $$
	
	\section{MUSIC algorithm}	
Let
	$$	A = - P^{-1}(k, \alpha) = 
	\begin{pmatrix}
		a_{11} & a_{12} & \cdots & a_{1M} \\
		a_{21} & a_{22} & \cdots & a_{2M} \\
		\vdots & \vdots & \ddots & \vdots \\
		a_{M1} & a_{M2} & \cdots & a_{MM}
	\end{pmatrix}
	\in \mathbb{C}^{M \times M}.
	$$
Define the matrix $ H \in \mathbb{C}^{M \times N} $, where $ H_{m,l} = e^{ik y_m \cdot \theta_l} $, for $ m = 1, 2, \ldots, M $ and $ l = 1, 2, \ldots, N $. Let $ H^* $ denote the conjugate transpose of $ H $, with $ H^* \in \mathbb{C}^{N \times M} $. Define the matrix $ F \in \mathbb{C}^{N \times N} $ by:
	
	$$
	F_{jl} := u^\infty(\theta_j, \theta_l) = \sum_{m,q=1}^M a_{mq} e^{ik (y_q \cdot \theta_l - \theta_j \cdot y_m)}.
	$$
	\begin{theorem} 
	The matrix $ F $ can be decomposed into: $ F = H^* A H $.
	\end{theorem}			
	\begin{proof}
		By the definition of $A$ and $H$, we have
		\begin{align*}
			 A H &=\begin{pmatrix}
			 		a_{11} & a_{12} & \cdots & a_{1M} \\
			 		a_{21} & a_{22} & \cdots & a_{2M} \\
			 		\vdots & \vdots & \ddots & \vdots \\
			 		a_{M1} & a_{M2} & \cdots & a_{MM}
			 		\end{pmatrix}
			 		\begin{pmatrix}
			 		e^{iky_1 \cdot \theta_1} & e^{iky_1 \cdot \theta_2} & \cdots & e^{iky_1 \cdot \theta_N} \\
			 		e^{iky_2 \cdot \theta_1} & e^{iky_2 \cdot \theta_2} & \cdots & e^{iky_2 \cdot \theta_N} \\
			 		\vdots & \vdots & \ddots & \vdots \\
			 		e^{iky_M \cdot \theta_1} & e^{iky_M \cdot \theta_2} & \cdots & e^{iky_M \cdot \theta_N}
					 \end{pmatrix} \\
				 &= \begin{pmatrix}
			 		\sum\limits_{q=1}^M a_{1q} e^{iky_q \cdot \theta_1} & \sum\limits_{q=1}^M a_{1q} e^{iky_q \cdot \theta_2} & \cdots & \sum\limits_{q=1}^M a_{1q} e^{iky_q \cdot \theta_N} \\
			 		\sum\limits_{q=1}^M a_{2q} e^{iky_q \cdot \theta_1} & \sum\limits_{q=1}^M a_{2q} e^{iky_q \cdot \theta_2} & \cdots & \sum\limits_{q=1}^M a_{2q} e^{iky_q \cdot \theta_N} \\
			 		\vdots & \vdots & \ddots & \vdots \\
			 		\sum\limits_{q=1}^M a_{Mq} e^{iky_q \cdot \theta_1} & \sum\limits_{q=1}^M a_{Mq} e^{iky_q \cdot \theta_2} & \cdots & \sum\limits_{q=1}^M a_{Mq} e^{iky_q \cdot \theta_N}
			 		\end{pmatrix},
		\end{align*}
		\begin{align*}
			H^* A H &= 
				\begin{pmatrix}
					e^{-iky_1 \cdot \theta_1} & e^{-iky_2 \cdot \theta_1} & \cdots & e^{-iky_M \cdot \theta_1} \\
					e^{-iky_1 \cdot \theta_2} & e^{-iky_2 \cdot \theta_2} & \cdots & e^{-iky_M \cdot \theta_2} \\
					\vdots & \vdots & \ddots & \vdots \\
					e^{-iky_1 \cdot \theta_N} & e^{-iky_2 \cdot \theta_N} & \cdots & e^{-iky_M \cdot \theta_N}
				\end{pmatrix}
				 \begin{pmatrix}
					\sum\limits_{q=1}^M a_{1q} e^{iky_q \cdot \theta_1} & \sum\limits_{q=1}^M a_{1q} e^{iky_q \cdot \theta_2} & \cdots & \sum\limits_{q=1}^M a_{1q} e^{iky_q \cdot \theta_N} \\
					\sum\limits_{q=1}^M a_{2q} e^{iky_q \cdot \theta_1} & \sum\limits_{q=1}^M a_{2q} e^{iky_q \cdot \theta_2} & \cdots & \sum\limits_{q=1}^M a_{2q} e^{iky_q \cdot \theta_N} \\
					\vdots & \vdots & \ddots & \vdots \\
					\sum\limits_{q=1}^M a_{Mq} e^{iky_q \cdot \theta_1} & \sum\limits_{q=1}^M a_{Mq} e^{iky_q \cdot \theta_2} & \cdots & \sum\limits_{q=1}^M a_{Mq} e^{iky_q \cdot \theta_N}
				\end{pmatrix}
			\\
			&= 
			\resizebox{\linewidth}{!}{$
				\begin{pmatrix}
					\sum\limits_{m,q=1}^M a_{mq} e^{ik(y_q \cdot \theta_1 - y_m \cdot \theta_1)} & \sum\limits_{m,q=1}^M a_{mq} e^{ik(y_q \cdot \theta_2 - y_m \cdot \theta_1)} & \cdots & \sum\limits_{m,q=1}^M a_{mq} e^{ik(y_q \cdot \theta_N - y_m \cdot \theta_1)} \\
					\sum\limits_{m,q=1}^M a_{mq} e^{ik(y_q \cdot \theta_1 - y_m \cdot \theta_2)} & \sum\limits_{m,q=1}^M a_{mq} e^{ik(y_q \cdot \theta_2 - y_m \cdot \theta_2)} & \cdots & \sum\limits_{m,q=1}^M a_{mq} e^{ik(y_q \cdot \theta_N - y_m \cdot \theta_2)} \\
					\vdots & \vdots & \ddots & \vdots \\
					\sum\limits_{m,q=1}^M a_{mq} e^{ik(y_q \cdot \theta_1 - y_m \cdot \theta_N)} & \sum\limits_{m,q=1}^M a_{mq} e^{ik(y_q \cdot \theta_2 - y_m \cdot \theta_N)} & \cdots & \sum\limits_{m,q=1}^M a_{mq} e^{ik(y_q \cdot \theta_N - y_m \cdot \theta_N)}
				\end{pmatrix}
				$}\\
			&=F.
		\end{align*}
	\end{proof}
Let $P$ be the orthogonal projection onto the range of $F$.
	\begin{theorem}
		Let $\{\theta_n : n \in \mathbb{N}\}$ be a countable set of directions satisfying the following prior condition: if an analytic function vanishes on all $\theta_n$, $n \in \mathbb{N}$, then the analytic function is identically zero on $\mathbb{S}^2$. Let $K$ be a compact subset of $\mathbb{R}^3$ containing all scatterers. For any sampling point $z \in \mathbb{R}^3$, define the test vector $\phi_z \in \mathbb{C}^N$ as:
			$$
			\phi_z = \big(e^{-ik\theta_1 \cdot z}, e^{-ik\theta_2 \cdot z}, \ldots, e^{-ik\theta_N \cdot z}\big)^T.
			$$
			Then there exists $N_0 \in \mathbb{N}$ such that for all $N \geq N_0$ and for all $z \in K$, we have
		\begin{align}
			z \in \{y_1, y_2, \ldots, y_M\} &\iff \phi_z \in Range(H^*) 
		\label{eq:first_iff} \\
			&\iff \phi_z \in Range(F)
	    \label{eq:second_iff} \\
			&\iff P\phi_z = 0 \notag
		\end{align}
		and $ \operatorname{Rank}(H) = M$.	
		\label{thm:important}
	\end{theorem} 
	\begin{proof}First, we prove \eqref{eq:first_iff}.
		
 If $ z \in \{y_1, y_2, \ldots, y_M\} $, then there exists $ m $, $ 1 \leq m \leq M $, such that $ z = y_m $. The test vector $ \phi_z = \phi_{y_m} $ corresponds to a column of $ H^* \in \mathbb{C}^{N \times M} $, and thus $ \phi_z \in Range(H^*) $.
		
The contrapositive of this conclusion is as follows: For any $ N \in \mathbb{N} $, there exists a point $ z \in K \setminus \{y_1, y_2, \ldots, y_M\} $ such that the set of vectors $ \{\phi_{y_1}, \phi_{y_2}, \ldots, \phi_{y_M}, \phi_z\} $ is linearly independent. We now prove that the contrapositive holds by contradiction.
			
Assume that for any $ N \in \mathbb{N} $, there exists a point $ z \in K \setminus \{y_1, y_2, \ldots, y_M\} $ such that the set of vectors $ \{\phi_{y_1}, \phi_{y_2}, \ldots, \phi_{y_M}, \phi_z\} $ is linearly dependent. Then there exists a sequence $ N_l \to \infty $, $ \{z^{(l)}\} \subset K \setminus \{y_1, y_2, \ldots, y_M\} $, $ \{\lambda^{(l)}\} \subset \mathbb{C}^M $, and $ \{\mu^{(l)}\} \subset \mathbb{C} $, satisfying
		\begin{align}
			|\mu^{(l)}| + \sum_{m=1}^M |\lambda_m^{(l)}| = 1, \quad \mu^{(l)} e^{-ik z^{(l)} \cdot \theta_j} + \sum_{m=1}^M \lambda_m^{(l)} e^{-ik y_m \cdot \theta_j} = 0, \quad \text{for all} \; j = 1, 2, \ldots, N_l.
		\label{eq:range_1}
		\end{align}
		
Since the sequences are bounded, as $ l \to \infty $, there exist convergent subsequences such that $ z^{(l)} \to z \in K $, $ \lambda^{(l)} \to \lambda \in \mathbb{C}^M $, and $ \mu^{(l)} \to \mu \in \mathbb{C} $. Fix $ j \in \mathbb{N} $, and let $ l \to \infty $. Then we have
		
		\begin{align}
			|\mu| + \sum_{m=1}^M |\lambda_m| = 1, \quad \mu e^{-ik z \cdot \theta_j} + \sum_{m=1}^M \lambda_m e^{-ik y_m \cdot \theta_j} = 0.
		\label{eq:range_2}
		\end{align}
		
Since the above equation holds for every $ j \in \mathbb{N} $, by the prior condition on $ \{\theta_j\} $, we have
		
		$$
		\mu e^{-ik z \cdot \theta} + \sum_{m=1}^M \lambda_m e^{-ik y_m \cdot \theta} = 0, \quad \text{for all} \; \theta \in \mathbb{S}^2.
		$$
	
The left-hand side of the above equation represents the far-field pattern of $ \mu \Phi_k(x, z) + \sum\limits_{m=1}^M \lambda_m \Phi_k(x, y_m) $. By Rellich's lemma and the unique continuation, we have
		
		$$
		\mu \Phi_k(x, z) + \sum_{m=1}^M \lambda_m \Phi_k(x, y_m) = 0, \quad \text{for all} \; x \notin \{z, y_1, \ldots, y_M\}.
		$$
We consider two cases:
	\begin{enumerate}
\item If $ z \notin \{y_1, y_2, \ldots, y_M\} $, letting $ x \to z $, yields $ \mu = 0 $. Letting $ x \to y_m $ for $ m = 1, 2, \ldots, M $, we obtain $ \lambda_m = 0 $ for $ m = 1, 2, \ldots, M $. This contradicts the first equation of \eqref{eq:range_2}.
	
\item If $z \in \{y_1, y_2, \ldots, y_M\}$. Without loss of generality, we assume $ z = y_1 $. Then
			$$
			(\mu + \lambda_1) e^{-ik y_1 \cdot \theta} + \sum_{m=2}^M \lambda_m e^{-ik y_m \cdot \theta} = 0, \text{for all} \; \theta \in \mathbb{S}^2.
			$$
			
The left-hand side of the above equation represents the far-field pattern of $ (\mu + \lambda_1)\Phi_k(x, y_1) + \sum\limits_{m=2}^M \lambda_m \Phi_k(x, y_m) $. By Rellich's lemma and the unique continuation, we have
			$$
			(\mu + \lambda_1) \Phi_k(x, y_1) + \sum_{m=2}^M \lambda_m \Phi_k(x, y_m) = 0.
			$$
Let $x \to y_1$ and $x \to y_m$, then
			\begin{align}
				\mu + \lambda_1 = 0, \quad \lambda_m = 0, \quad m = 2, 3, \ldots, M.
			\label{eq:range_3}
			\end{align}
			
Rewriting \eqref{eq:range_1}, we obtain
			
		\begin{align}
			\begin{split}
				[\mu^{(l)} + \lambda_1^{(l)}] e^{-iky_1 \cdot \theta_j} + \mu^{(l)} [e^{-ikz^{(l)} \cdot \theta_j} - e^{-iky_1 \cdot \theta_j}] + \sum_{m=2}^M \lambda_m^{(l)} e^{-iky_m \cdot \theta_j} = 0, \\
				\text{for all } j = 1, 2, \ldots, N_l.
			\end{split}
			\label{eq:range_4}
		\end{align}
Let $\rho_l := |\mu^{(l)} + \lambda_1^{(l)}| + \sum\limits_{m=2}^M |\lambda_m^{(l)}| + |z^{(l)} - y_1|,$ which converges to zero as $l \to \infty$. By the Taylor formula, we have
			$$
			e^{-ikz^{(l)} \cdot \theta_j} - e^{-iky_1 \cdot \theta_j} = -ik\theta_j (z^{(l)} - y_1) e^{-iky_1 \cdot \theta_j} + o(|z^{(l)} - y_1|^2).
			$$
Let
			$$
			\tilde{\lambda}_1^{(l)} = \frac{\mu^{(l)} + \lambda_1^{(l)}}{\rho_l}, \quad \tilde{\lambda}_m^{(l)} = \frac{\lambda_m^{(l)}}{\rho_l}, \quad m=2,3,\ldots,M, \quad a^{(l)} = \frac{z^{(l)} - y_1}{\rho_l}.
			$$
			
For all $j=1,2,\ldots,N_l$,
			$$
			\left[ \tilde{\lambda}_1^{(l)} - ik \mu^{(l)} \theta_j a^{(l)} \right] e^{-iky_1 \cdot \theta_j} + \sum_{m=2}^M \tilde{\lambda}_m^{(l)} e^{-iky_m \cdot \theta_j} = o(|z^{(l)} - y_1|).
			$$
These sequences are also bounded, so we can find convergent subsequences. Taking $l \to \infty$, $\tilde{\lambda}_m^{(l)} \to \tilde{\lambda}_m$, $m=1,2,\ldots,M$, $a^{(l)} \to a$, we have
			\begin{align}
				\sum_{m=1}^{M} |\tilde{\lambda}_m| + |a| = 1,
			\label{eq:range_5}
			\end{align}
			$$
			[\tilde{\lambda}_1 - ik \mu \theta_j a] e^{-ik y_1 \cdot \theta_j} + \sum_{m=2}^M \tilde{\lambda}_m e^{-ik y_m \cdot \theta_j} = 0, \quad \text{for all} \; j \in \mathbb{N}.
			$$
By the prior condition on $ \theta_j $, we have
			\begin{align}
				[\tilde{\lambda}_1 - ik \mu \theta a ] e^{-ik y_1 \cdot \theta} + \sum_{m=2}^{M} \tilde{\lambda}_m e^{-ik y_m \cdot \theta} = 0, \quad \text{for all} \; \theta \in \mathbb{S}^2.
			\label{eq:range_6}
			\end{align}
The left-hand side of the above equation represents the far-field pattern of 
           $$ \tilde{\lambda}_1 \Phi_k(x, y_1) + \mu a \nabla_y \Phi_k(x, y_1) + \sum\limits_{m=2}^M \tilde{\lambda}_m \Phi_k(x, y_m). $$
By Rellich's lemma and the unique continuation, we have
			\begin{align*}
				\tilde{\lambda}_1 \Phi_k(x, y_1) + \mu a \nabla_y \Phi_k(x, y_1) + \sum_{m=2}^M \tilde{\lambda}_m \Phi_k(x, y_m) = 0, \quad \quad \text{for all} \; x \notin \{y_1, y_2, \ldots, y_M\}.
			\end{align*}
Letting $x \to y_1$ and $x \to y_m$, $m=2, 3, \ldots, M$, we can get $\tilde{\lambda}_m = 0$ when $m=2, 3, \ldots, M$. When $x \neq y_1$, we have
			\begin{align*}
				\tilde{\lambda}_1 \Phi_k(x, y_1) + \mu a \nabla_y \Phi_k(x, y_1) 
				&= \left[ \tilde{\lambda}_1 + \mu a \frac{y_1 - x}{|y_1 - x|} (ik - \frac{1}{|y_1 - x|}) \right] \frac{e^{ik |y_1 - x|}}{4\pi |y_1 - x|}. \\
				&= 0.
			\end{align*}
It is easy to see that $ \tilde{\lambda}_1 = 0 $ and $ \mu a = 0 $. From \eqref{eq:range_2} and \eqref{eq:range_3}, we know that $ |\mu| = \frac{1}{2} $, and thus $ a = 0 $, which contradicts \eqref{eq:range_5}. Therefore, the assumption does not hold, i.e., for any $ N \in \mathbb{N} $, there exists a point $ z \in K \setminus \{y_1, y_2, \ldots, y_M\} $ such that the set of vectors $ \{\phi_{y_1}, \phi_{y_2}, \ldots, \phi_{y_M}, \phi_z\} $ is linearly independent. That is, there exists $ N_0 \in \mathbb{N} $ such that for all $ N \geq N_0 $ and for all $ z \in K $, if $ \phi_z \in Range (H^*) $, then $ z \in \{y_1, y_2, \ldots, y_M\} $.
	\end{enumerate}		
Next, we prove that the rank of $ H $ is $ M $.

Define $ \beta_m $ as follows 
	
			$$
			\beta_m: = \left( e^{ik y_m \cdot \theta_1}, e^{ik y_m \cdot \theta_2}, \ldots, e^{ik y_m \cdot \theta_N} \right)^T, \quad m=1,2,\ldots,M,
			$$
where $ \beta_m $ is a column vector of $ H^T $.
We now prove that $ \beta_m $, $ m = 1, 2, \ldots, M $, are linearly independent. We assume on the contrary that $ \beta_m $, $ m = 1, 2, \ldots, M $, are linearly dependent. Then there exist non-vanishing coeffients $ b_m \in \mathbb{C} $, $ m = 1, 2, \ldots, M $, such that
			$$
			\sum_{m=1}^M b_m \beta_m = 0,
			$$
i.e.,
			$$
			\sum_{m=1}^M b_m e^{ik y_m \cdot \theta_j} = 0, \quad \quad \quad \quad \quad \text{for all} \; j=1,2,\ldots,N.
			$$
Let $\hat{x} := \frac{x}{|x|} \in \mathbb{S}^2$ and define $f(\hat{x}) = \sum\limits_{m=1}^M b_m e^{ik y_m \cdot \hat{x}}$. The function $ f(\hat{x}) $ is analytic on $ \mathbb{S}^2 $. Since $ f(\theta_j) = 0 $ for $ j = 1, 2, \ldots, N $, by the prior condition on $ \theta_j $, we have		
			$$
			f(\hat{x}) = \sum_{m=1}^M b_m e^{ik y_m \cdot \hat{x}} = 0, \quad \quad \quad \quad \quad \text{for all} \; \hat{x} \in \mathbb{S}^2.
			$$
Taking the conjugate on both sides, we have
			
			$$
			\overline{f(\hat{x})} = \sum_{m=1}^M \overline{b_m} e^{-ik y_m \cdot \hat{x}} = 0.
			$$
From the above equation, we know that the far-field pattern of $ \sum\limits_{m=1}^M \overline b_m \Phi_k(x, y_m) $ is zero. By Rellich's lemma and unique continuation, we have
			$$
			\sum_{m=1}^M \overline{b_m} \Phi_k(x, y_m) = 0.
			$$
Let $ x \to y_m $, $ m = 1, 2, \ldots, M $. Then we have $\overline b_m = 0 $, $ m = 1, 2, \ldots, M $, which implies $ b_m = 0 $, $ m = 1, 2, \ldots, M $. This is a contradiction. Therefore, $ \beta_m $, $ m = 1, 2, \ldots, M $, are linearly independent. Since $ N \geq M $, the rank of $ H^T $ is $ M $, and thus the ranks of $ H $ and $ H^* $ are also $ M $.
	
Finally, we prove \eqref{eq:second_iff}, i.e.,  Range $(H^*)$ = Range $(F) $.

    For any $ \psi \in$ Range$(H^*) $, there exists $ \varphi \in \mathbb{C}^M $ such that $ H^* \varphi = \psi $. Since $ A $ is an invertible matrix, there exists $ \eta \in \mathbb{C}^M $ such that $ \eta = A^{-1} \varphi \in \mathbb{C}^M $. Since the rank of $ H $ is $ M $, we have  Range$(H) = \mathbb{C}^M $. Thus, there exists $ \gamma^* \in \mathbb{C}^N $ such that $ H \gamma^* = \eta $. Then
			$$
			F(\gamma^*) = H^* A H (\gamma^*) = H^* A (\eta) = H^* \varphi = \psi.
			$$
	Hence $\psi \in$ Range$(F)$, i.e. Range$(H^*) \subseteq$ Range$(F)$.\\
For any $ \psi \in$ Range$(F) $, there exists $ \varphi \in \mathbb{C}^N $ such that $ F ( \varphi )= \psi $. Since $ F = H^* A H $, we have $ \psi = H^* (A H \varphi) $, which implies $ \psi \in$ Range$(H^*) $. Thus,  Range $(F) \subseteq$ Range$(H^*) $.
We conclude from above that Range $(F)$ = Range $(H^*) $.
	
	Since $ P : \mathbb{C}^N \to$ Range $(F)^\perp = N(F^*) $ is an orthogonal projection, $ \phi_z \in $Range$(F) $ is equivalent to $ P \phi_z = 0 $. There exists $ N_0 \in \mathbb{N} $ such that for all $ z \in K $, 
	$$
	z \in \{y_1, y_2, \ldots, y_M\} \iff \phi_z \in Range(H^*) \iff \phi_z \in Range(F) \iff P \phi_z = 0.
	$$
	\end{proof}
	\begin{theorem} Assume the conditions in Theorem \ref{thm:important} are fulfilled. Define the test function  $I_z = \frac{1}{|P \phi_z|}$. Then we have
		$$
		z \in \{y_1, y_2, \ldots, y_M\} \iff I_z = +\infty. 
		$$
	\end{theorem} 
	\begin{proof}
	If $ z \in \{y_1, y_2, \ldots, y_M\} $, there exists $ 1 \leq m \leq M $ such that $ z = y_m $. From Theorem \ref{thm:important}, we know that $ \phi_z \in $Range $(F) $, which implies $ P \phi_z = 0 $. Therefore, $ I_z = +\infty $. 
		
	If $ I_z = +\infty $, we have $ P \phi_z = 0 $. By Theorem \ref{thm:important}, it follows that $ z \in \{y_1, y_2, \ldots, y_M\} $.
	\end{proof}

\section{Numerical Tests}
 We focus on solving the problem of point source localization using MATLAB. We first perform Singular Value Decomposition (SVD) on the matrix $F$ to obtain the orthonormal basis $q_i$ of the range space of $F$,  where $i = 1, 2, \dots, \text{rank}(F)$. We definethe following test function
			$$
				I_z = \frac{1}{\left| \phi_z - (\phi_z, q_i) \cdot q_i \right|}.
			$$

To validate the effectiveness of the algorithm, we conducted numerical experiments with three point sources. The positions of these point sources were set as $(5, 0)$, $(-5, 0)$ and $(3, 9)$. In numerics, we divided the space with a step size of $0.1$, used $20$ observation directions, set the wave number $k = 2\pi$, and let $\alpha = (1+i, 3+5i, -1+5i)$.

\subsection{Inversion Results from Noise-Free data}

In the absence of noise, the inversion results are shown in Figure A. In the figure, the green dots represent the point source positions obtained from theoretical inversion, while the red circles indicate the pre-defined point source positions. From the image, it is clear that the predicted point source positions perfectly overlap with the true positions. This demonstrates that the inversion algorithm based on SVD decomposition can accurately localize point sources in an ideal noise-free environment.

\subsection*{Inversion with Low Noise data}

We introduce a complex-valued additive noise model to simulate measurement errors in the data. The noisy data matrix $\mathbf{F}_{\text{noisy}}$ is constructed as follows
 		$$
 	      \mathbf{F}_{\text{noisy}} = \mathbf{F} + \delta \cdot \max_{i,j} |\mathbf{F}_{ij}| \cdot \left( U(-1, 1)^{N \times N} + i \cdot U(-1, 1)^{N \times N} \right),
	    $$
 where $U(-1, 1)$ represents the uniform distribution over the interval $[-1, 1]$, and $\delta> 0$ is a prescribed noise factor that determines the intensity of the perturbation relative to the largest absolute entry in $\mathbf{F}$.When noise is added to the data, the situation changes significantly. For example, when $\delta$ is $0.1\%$, the inversion results are shown in Figure B. Surprisingly, even with such a small amount of noise, the algorithm fails to correctly invert the point source positions. Through extensive testing and analysis, we identified the root cause: when small perturbations are introduced, the results of the SVD decomposition of matrix $F$ change drastically. This leads to large values of the test function $I_z$ across the entire test region, causing all positions in the space to potentially be identified as predicted point source locations, thereby resulting in inversion failure.

\subsection*{Improved Algorithm Based on Pseudo-Inverse}

To address the above noise sensitivity issue, we improved the algorithm using pseudo-inverse. Under the same conditions, we used the pseudo-inverse of matrix $F$ to compute the test function $I_z$ . 

We define the orthogonal projection matrix $\mathbf{Q}_{\text{proj}} \in \mathbb{C}^{N \times N}$ onto the left null space of $\mathbf{F} \in \mathbb{C}^{M \times N}$ as:
$$
\mathbf{Q}_{\text{proj}} = \mathbf{I}_N - \mathbf{F} \mathbf{F}^\dagger,
$$
where $\mathbf{F}^\dagger$ denotes the Moore--Penrose pseudoinverse of $\mathbf{F}$, and $\mathbf{I}_N$ is the $N \times N$ identity matrix.

For a sampling point $ z \in \mathbb{R}^3 $, we define the test vector $ \phi_z \in \mathbb{C}^N $ component-wise by
$$
	\phi_z = \big(e^{-ik\theta_1 \cdot z}, e^{-ik\theta_2 \cdot z}, \ldots, e^{-ik\theta_N \cdot z}\big)^T,
$$
where $\theta_j \in Q ,j =1 ,2 ,\cdots   ,N.$
Based on this construction, we define the test function $ I_z $ as follows:
	$$   
		I_z = \frac{1}{\left\| \mathbf{Q}_{\text{proj}} \phi_z \right\|},
	$$
where $ \|\cdot\| $ denotes the standard Euclidean norm. The function $ I_z $ achieves its maximum at the locations where $ \phi_z $ aligns with the range of $ \mathbf{F} $, and thus can be used to identify point sources or scatterers.

In the noise-free case, the results generated by the improved algorithm are shown in Figure C. As can be seen from the figure, the improved algorithm still performs well in localizing point sources, indicating that the pseudo-inverse-based computation method also exhibits good performance in noise-free environments.
 
\subsection*{Inversion Performance with High Noise data}

To further verify the robustness of the improved algorithm, we add noise to the new model using the same noise addition method described above. We added $20\%$ noise to the data. The inversion results under this condition are shown in Figure D. Despite serious noise pollution, it is evident from the image that the improved algorithm still successfully localizes the point sources. This demonstrates that using the pseudo-inverse of matrix $F$ to compute the test function significantly enhances the stability and reliability of the algorithm in noisy environments.
	
	\begin{figure}[htbp]
		\centering
		\begin{subfigure}[b]{0.48\textwidth}
			\centering
			\includegraphics[width=\textwidth]{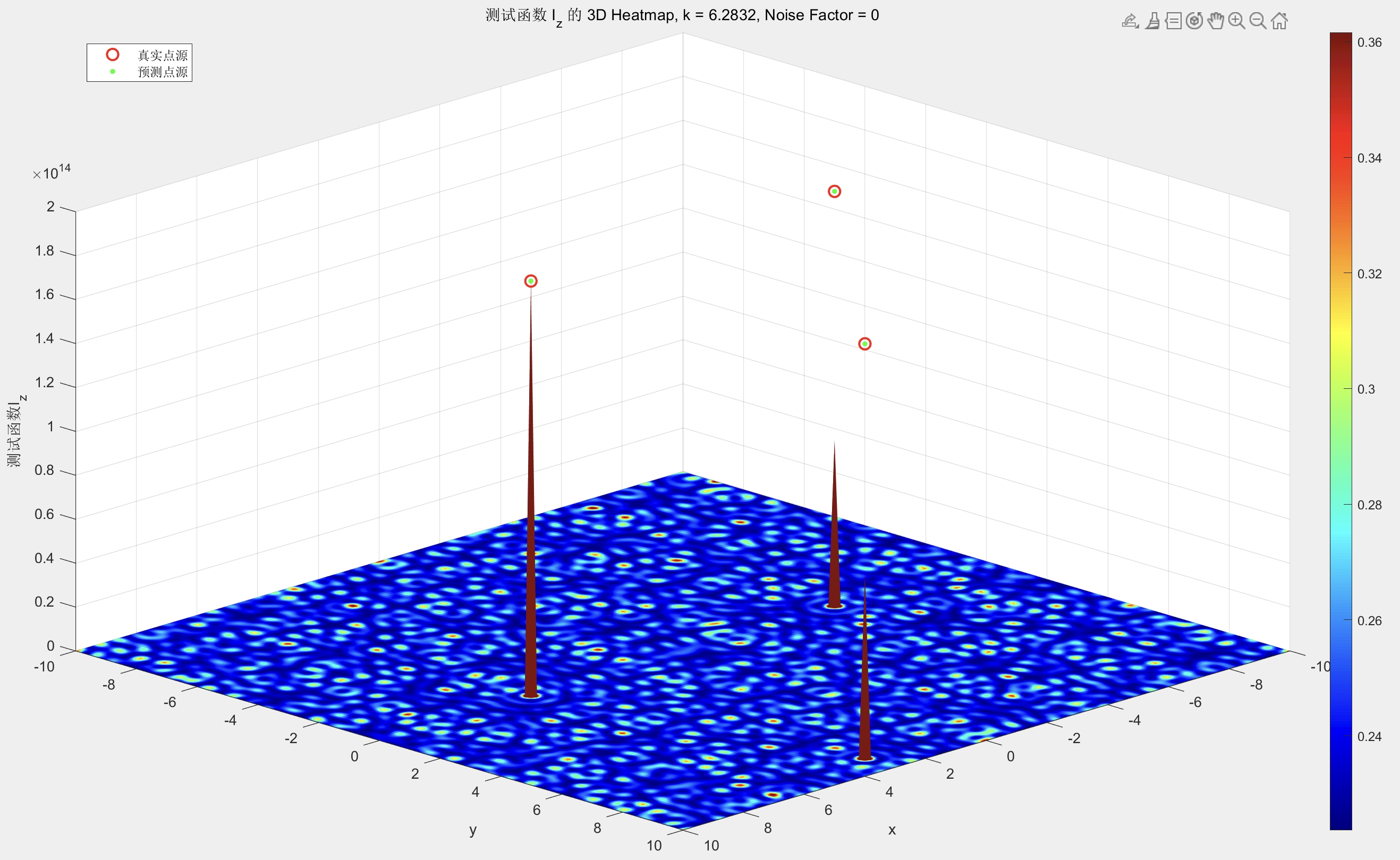}
			\caption{3D Heatmap of $ I_z $, $\delta$ = 0.}
			\label{fig:heatmap_no_noise}
		\end{subfigure}
		\hfill
		\begin{subfigure}[b]{0.48\textwidth}
			\centering
			\includegraphics[width=\textwidth]{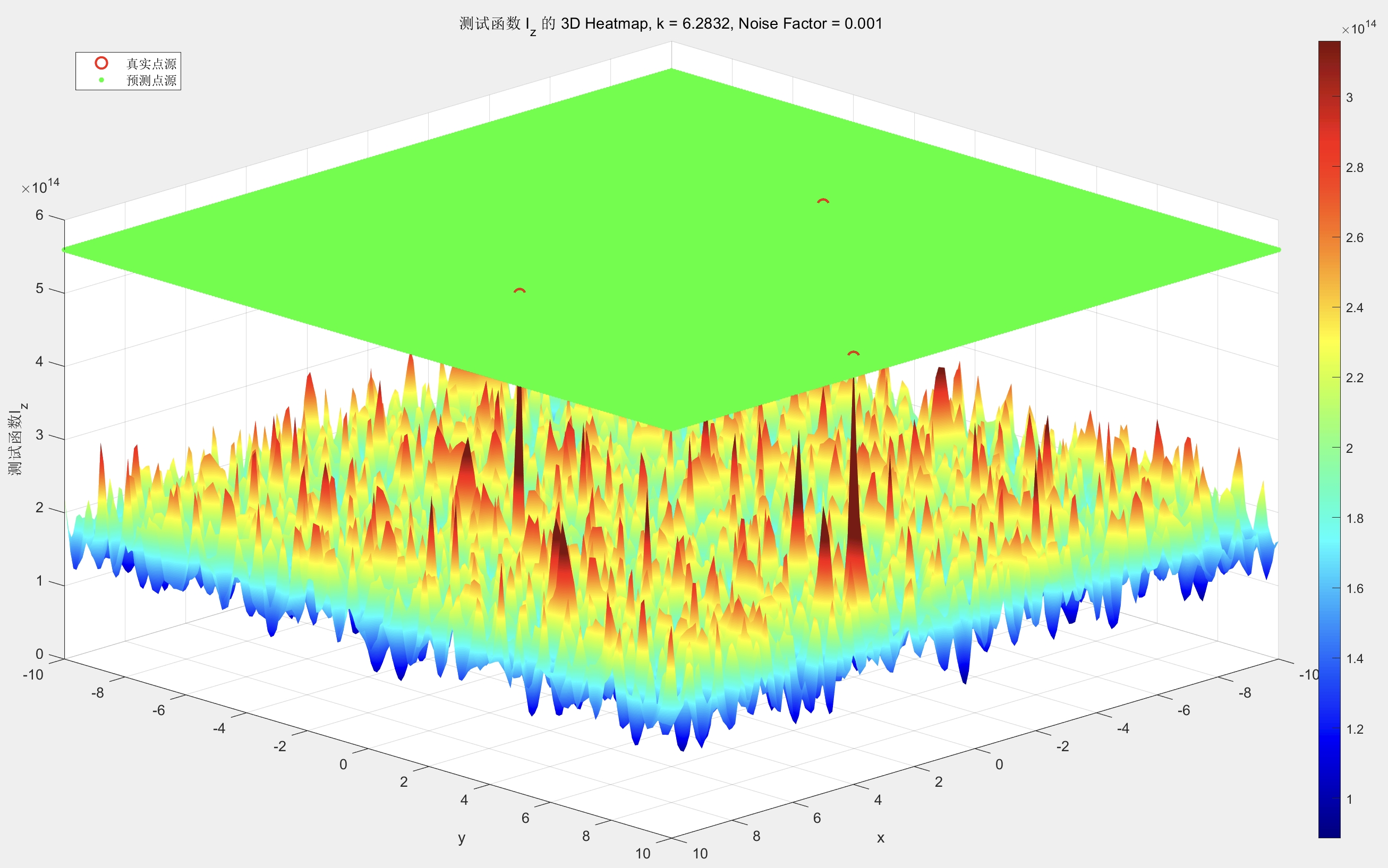}
			\caption{3D Heatmap of $ I_z $, $\delta$ = 0.001.}
			\label{fig:heatmap_with_noise}
		\end{subfigure}
		
		\begin{subfigure}[b]{0.48\textwidth}
			\centering
			\includegraphics[width=\textwidth]{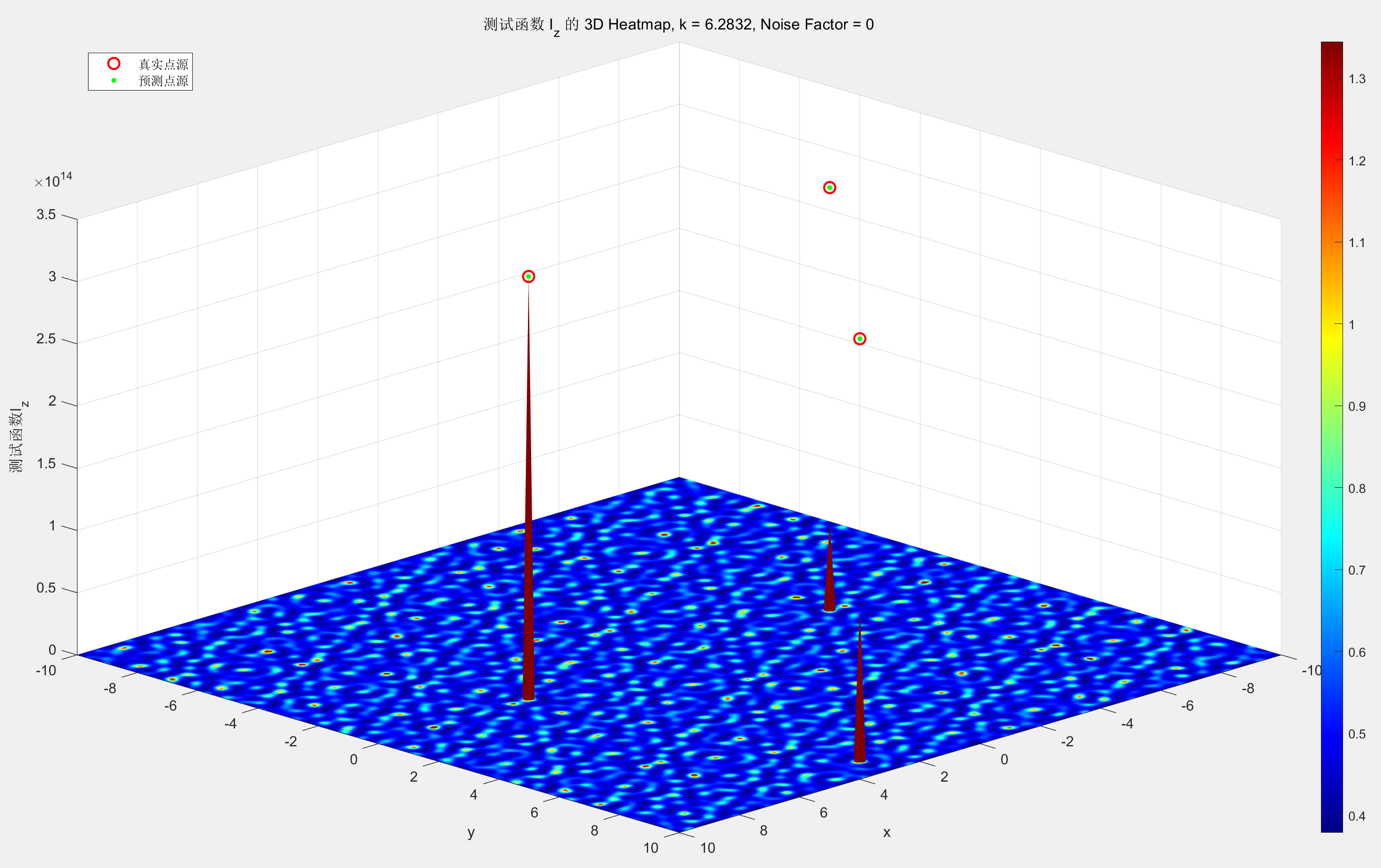}
			\caption{3D Heatmap of $ I_z $, using the pseudo-inverse of $ F $, $\delta$ = 0.}
			\label{fig:heatmap_pseudo_inverse}
		\end{subfigure}
		\hfill
		\begin{subfigure}[b]{0.48\textwidth}
			\centering
			\includegraphics[width=\textwidth]{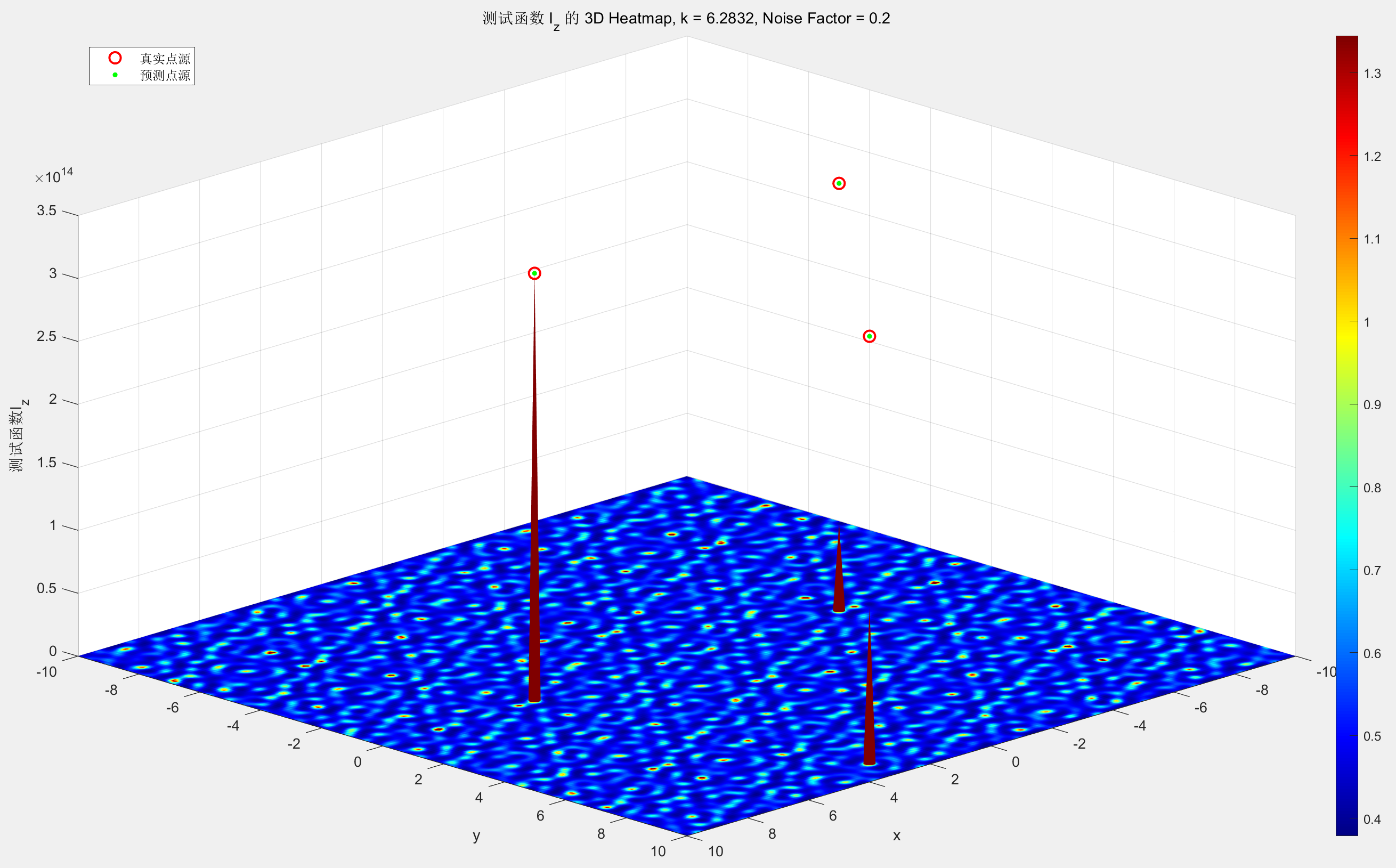}
			\caption{3D Heatmap of $ I_z $, $\delta$ = 0.2.}
			\label{fig:heatmap_with_noise_02}
		\end{subfigure}
		\caption{Comparison of 3D Heatmap results for $ I_z $, $k=6.2832$.}
		\label{fig:heatmaps}
	\end{figure}

In summary, using the pseudo-inverse of $F$, we effectively addressed the noise sensitivity issue of the original algorithm. The improved algorithm maintains good inversion performance under different noise levels, providing a more reliable solution for point source localization inversion in practical applications.

\subsection{}
To verify the effectiveness of the bar algorithm, we conduct five additional experiments: varying the number of point sources, different observation directions, different wave numbers, using different $\alpha$ values, and changing the noise levels. The detailed results are presented below, demonstrating the effectiveness and stability of the proposed algorithm.

\subsection*{Reconstruction of different numbers of point sources}

To validate the effectiveness of the algorithm, we conducted numerical experiments with four, five and six point sources. During the experiment, we divided the space with a step size of $0.1$, used $20$ observation directions, set the wave number $k = 2\pi$, used $\delta = 0.2$ and let $\alpha = [1+1i; 3+5i; -1+5i; i]$, $\alpha = [1+1i; 3+5i; -1+5i; i; -2+7i]$,  $\alpha = [1+1i; 3+5i; -1+5i; i; -2+7i; 6+3i]$ .
The positions of the point sources are as follows:
\begin{itemize}
	\item The first set of point source positions : $(3, -2), ( 5, 3), (-7, 9), (4, 8)$.
	\item The second set of point source positions : $(3, -2), (5, 3), (-7, 9), (4, 8), (-3,-2)$.
	\item The third set of point source positions : $(3, -2), (5, 3), (-7, 9), (4, 8), (-3,-2), (7, 8)$.
\end{itemize}
The generated images are shown in Figure~2.
\begin{figure}[ht]
			\centering
			\begin{subfigure}[b]{0.32\textwidth}
				\includegraphics[width=\linewidth]{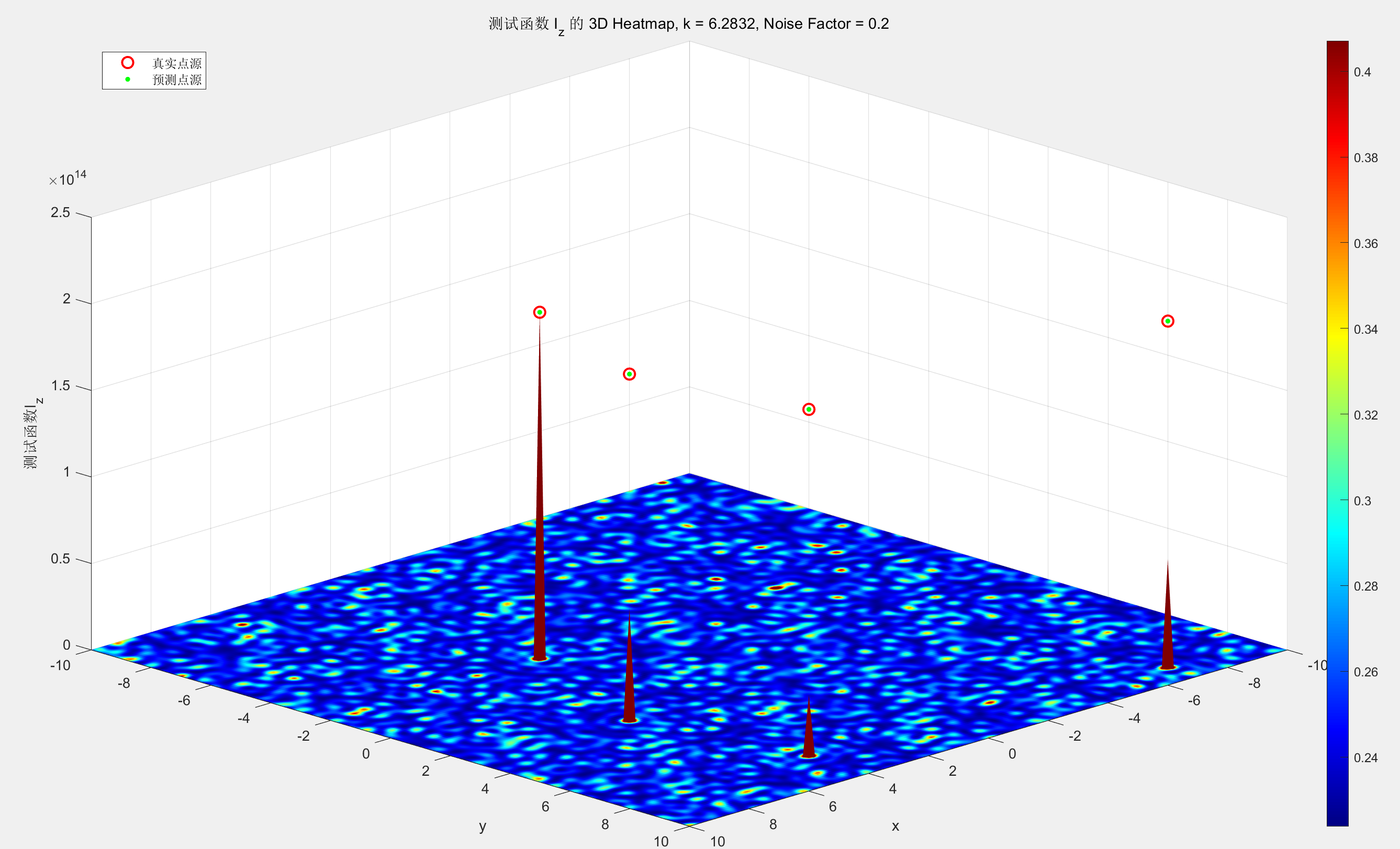}
				\caption{Four point sources}
				\label{fig:fourpointsources}
			\end{subfigure}
			\hfill
			\begin{subfigure}[b]{0.32\textwidth}
				\includegraphics[width=\linewidth]{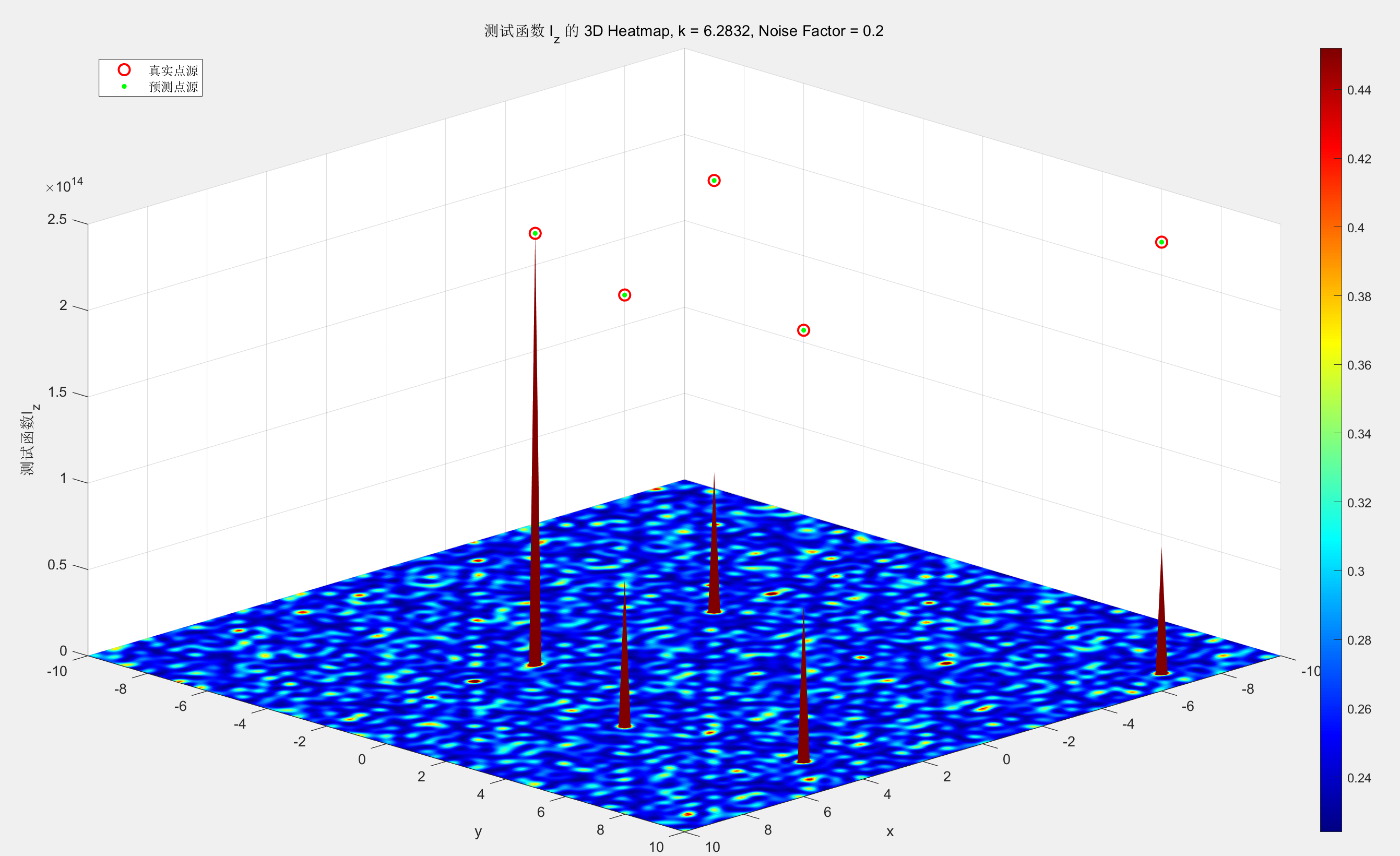}
				\caption{Five point sources}
				\label{fig:pointsourcesB}
			\end{subfigure}
			\hfill
			\begin{subfigure}[b]{0.32\textwidth}
				\includegraphics[width=\linewidth]{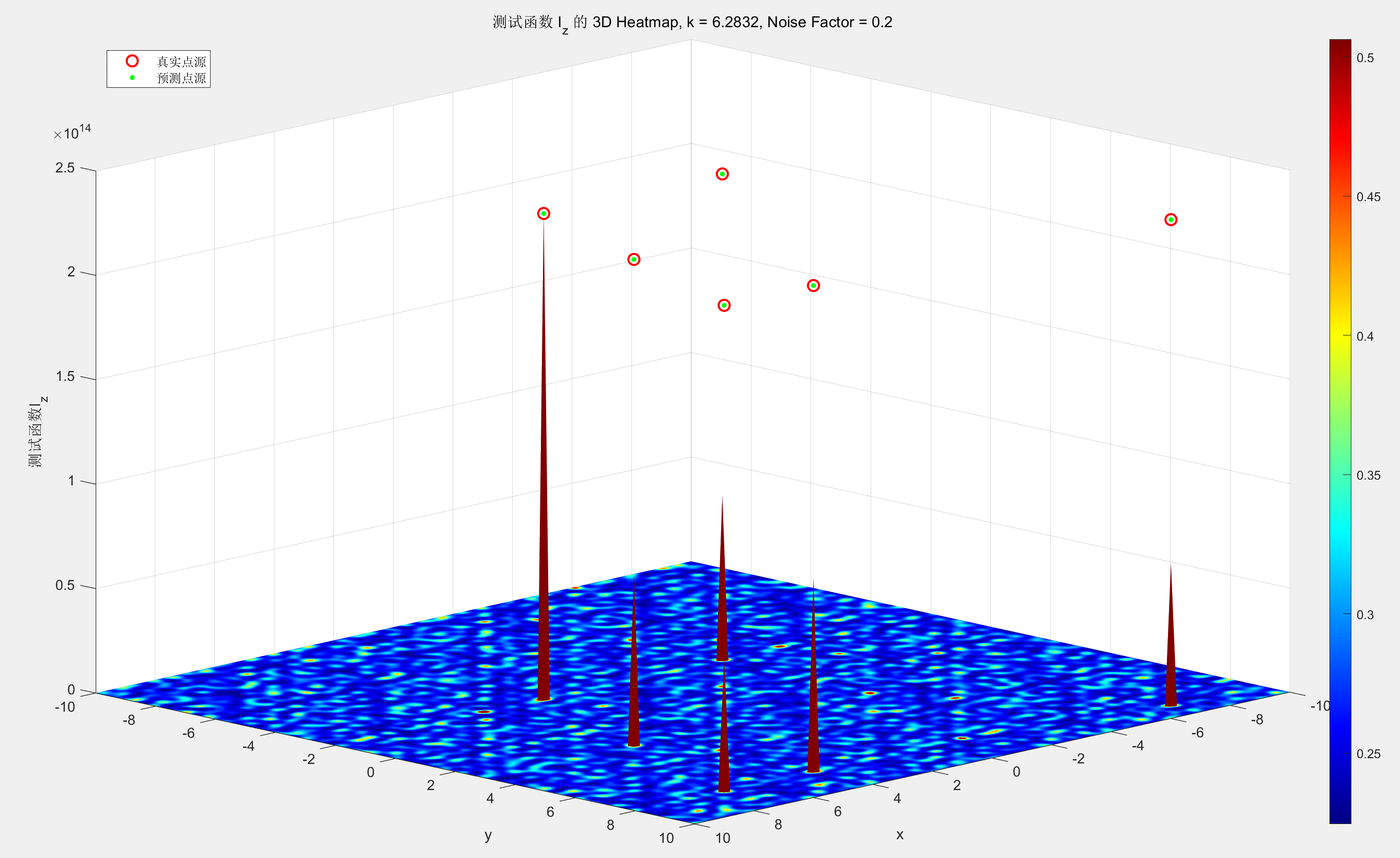}
				\caption{Six point sources}
				\label{fig:pointsourcesC}
			\end{subfigure}
			\caption{Reconstruction of different numbers of point sources.}
			\label{fig:numberofpointsources}
\end{figure}

From the images, it can be observed that when the number of true point sources is changed, the predicted point source locations still align well with the true point sources.
	
\subsection*{Using different numbers of observation directions}
To validate the effectiveness of the algorithm, we conducted numerical experiments with 30, 40, 50 observation directions. During the experiment, we divided the space with a step size of $0.1$, set the wave number $k = 2\pi$, used $\delta = 0.2$ and let $\alpha = [1+1i; 3+5i; -1+5i; i; -2+7i; 6+3i]$ .The positions of these point sources were set as $(3, -2), (5, 3), (-7, 9), (4, 8), (-3, -2), (7, 8)$.
The generated images are shown in Figure~3.
\begin{figure}[ht]
	\centering
	\begin{subfigure}[b]{0.32\textwidth}
		\includegraphics[width=\linewidth]{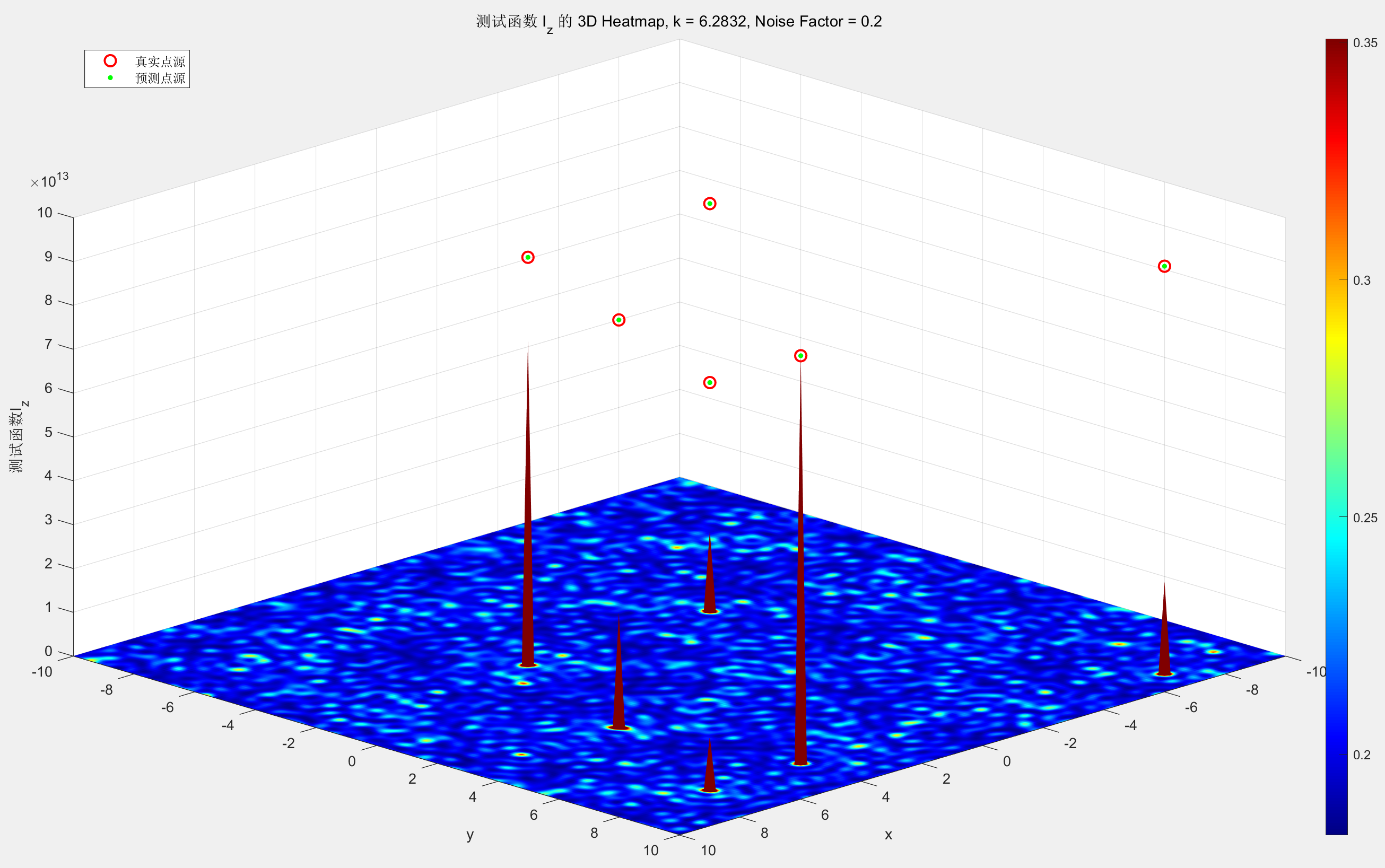}
		\caption{30 observation directions }
		\label{fig:observation directions 30.png}
	\end{subfigure}
	\hfill
	\begin{subfigure}[b]{0.32\textwidth}
		\includegraphics[width=\linewidth]{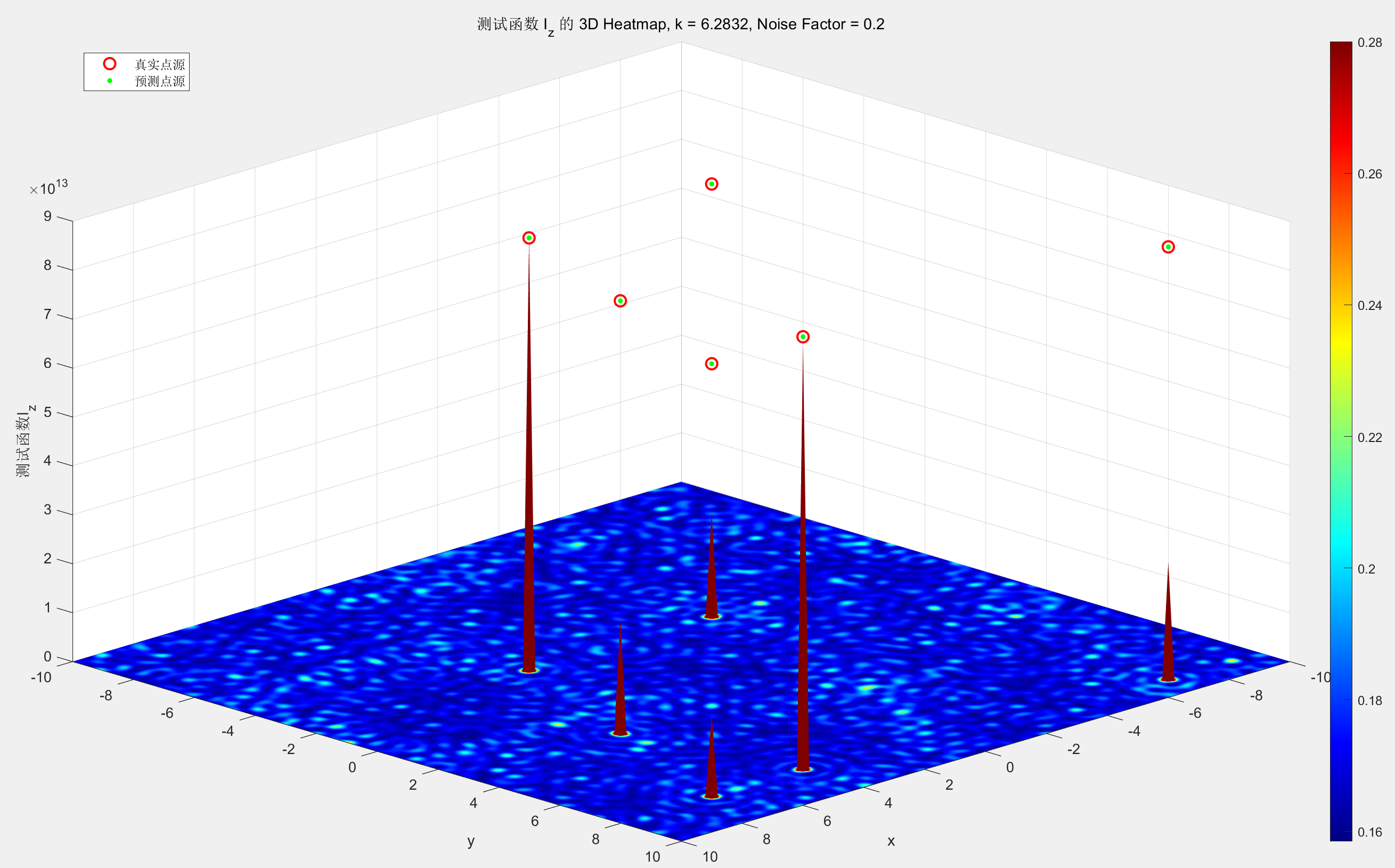}
		\caption{40 observation directions}
		\label{fig:observation directions 40.png}
	\end{subfigure}
	\hfill
	\begin{subfigure}[b]{0.32\textwidth}
		\includegraphics[width=\linewidth]{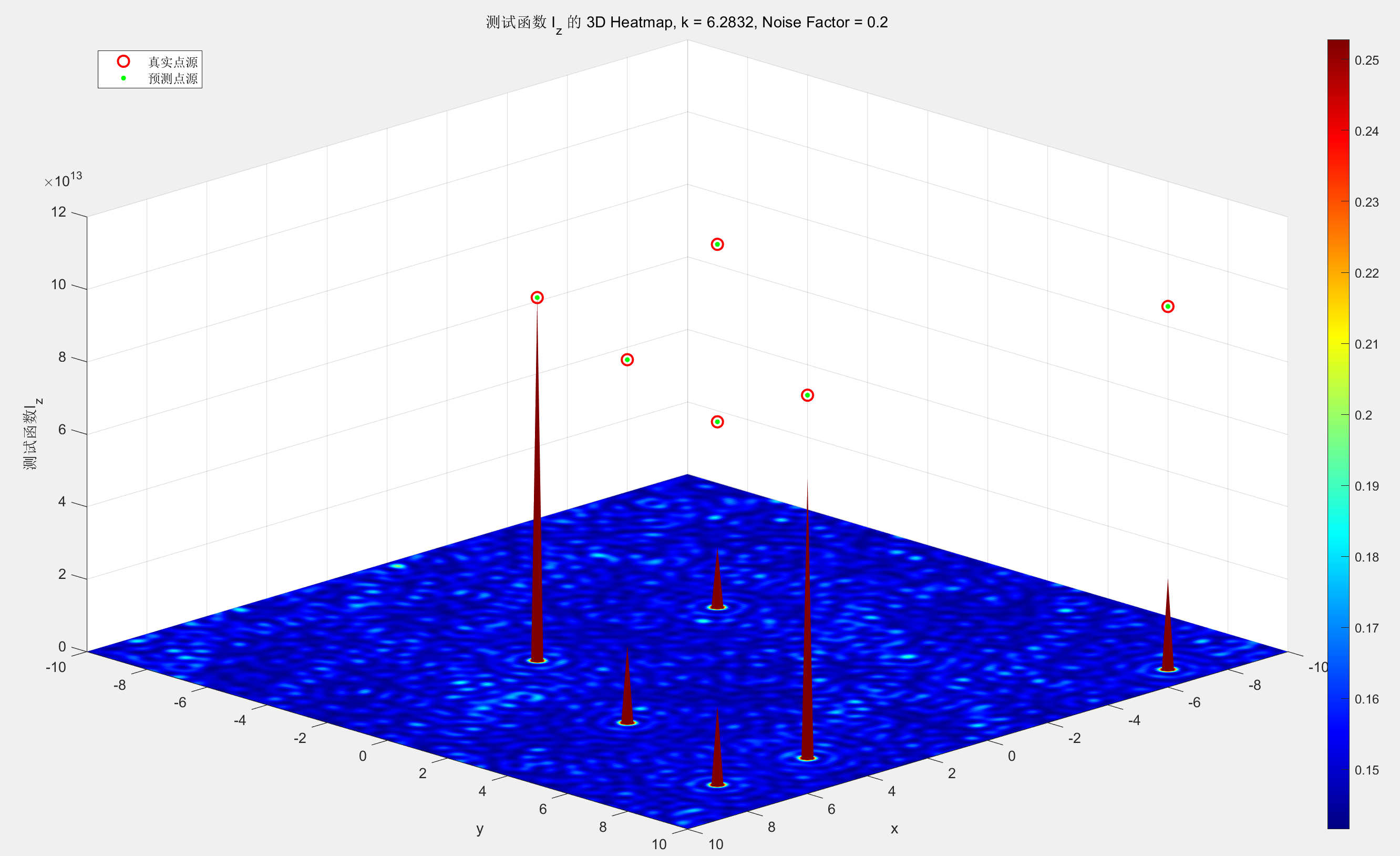}
		\caption{50 observation directions}
		\label{fig:observation directions 50.png}
	\end{subfigure}
	\caption{Using different numbers of observation directions.}
	\label{fig:observation directions}
\end{figure}

From the numerical experiments, it can be seen that when the number of arbitrarily selected observation directions exceeds the number of point sources, the inversion of the point source locations can be carried out normally.

\subsection*{Using different wave numbers}
To validate the effectiveness of the algorithm, we conducted numerical experiments with $k =$ $\pi$, $3\pi$, $4\pi$. During the experiment, we divided the space with a step size of $0.1$, set observation directions number $N = 20$, used $\delta = 0.2$ and let $\alpha = [1+1i; 3+5i; -1+5i; i; -2+7i; 6+3i]$ .The positions of these point sources were set as $(3, -2), (5, 3), (-7, 9), (4, 8), (-3, -2), (7, 8)$.
The generated images are shown in Figure~4.

\begin{figure}[ht]
	\centering	
	\begin{subfigure}[b]{0.32\textwidth}
		\includegraphics[width=\linewidth]{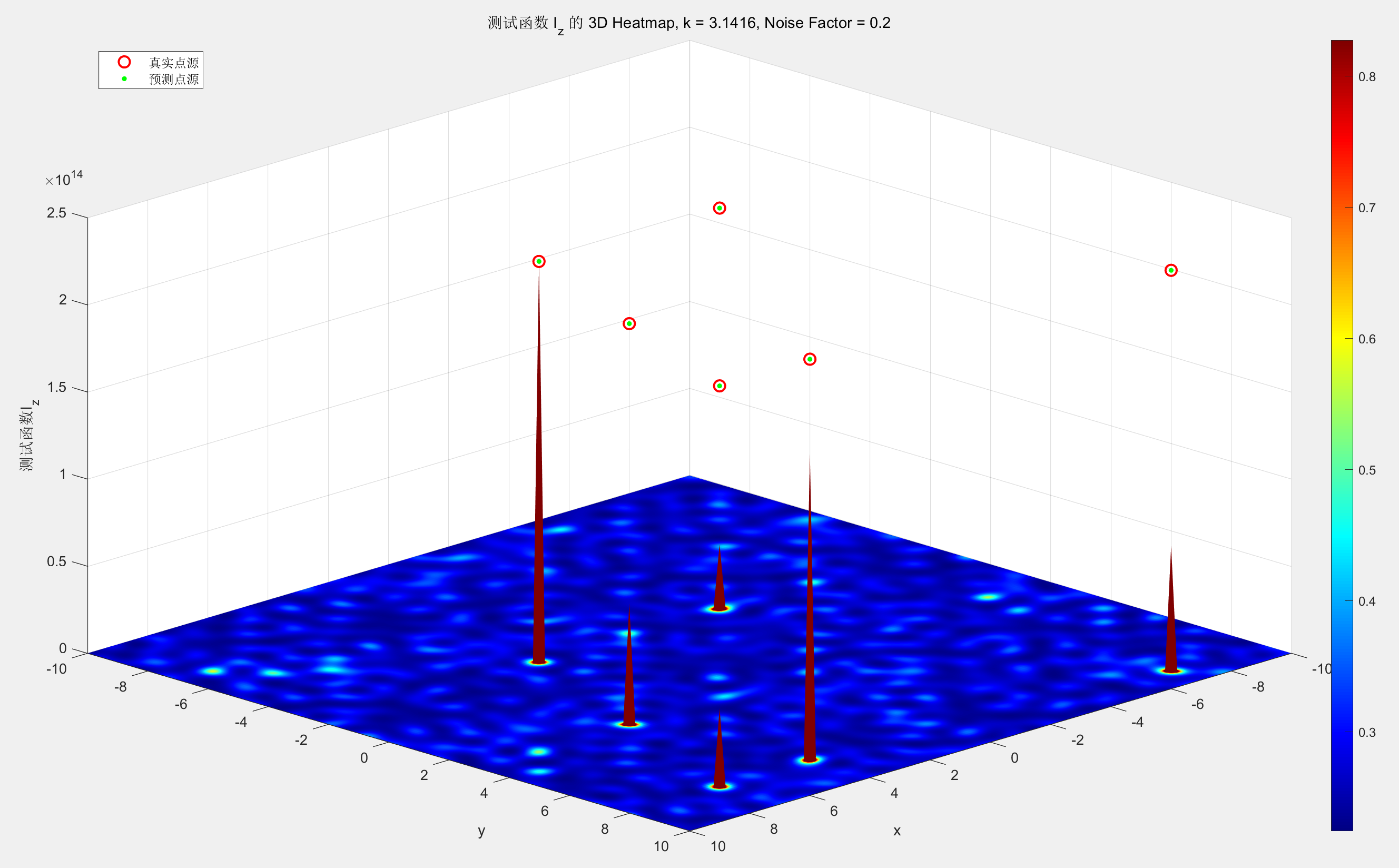}
		\caption{$k= \pi$}
		\label{fig:wave number pi.png}
	\end{subfigure}
	\begin{subfigure}[b]{0.32\textwidth}
		\includegraphics[width=\linewidth]{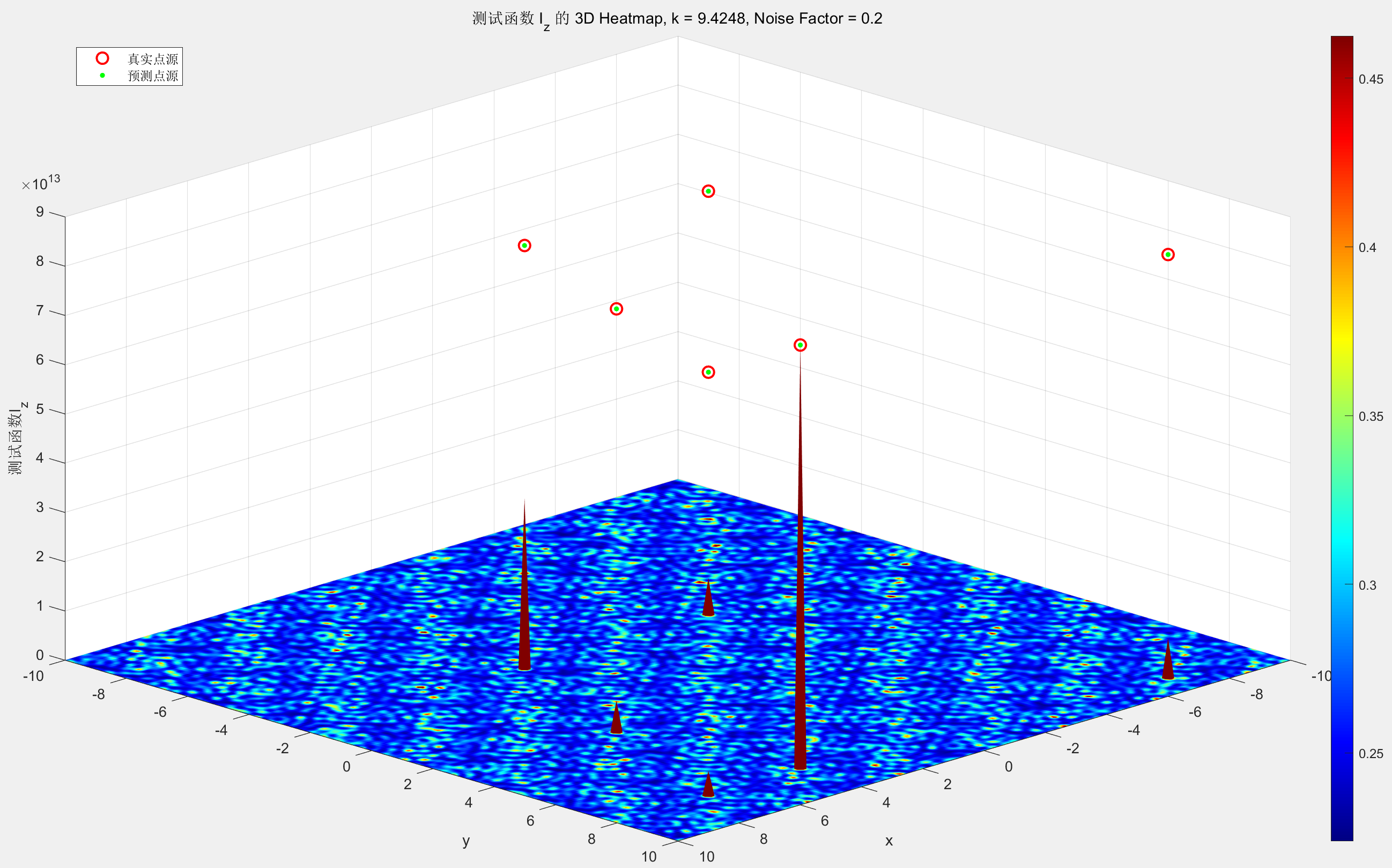}
		\caption{$k= 3\pi$}
		\label{fig:wave number 3pi.png}
	\end{subfigure}
	\hfill
	\begin{subfigure}[b]{0.32\textwidth}
		\includegraphics[width=\linewidth]{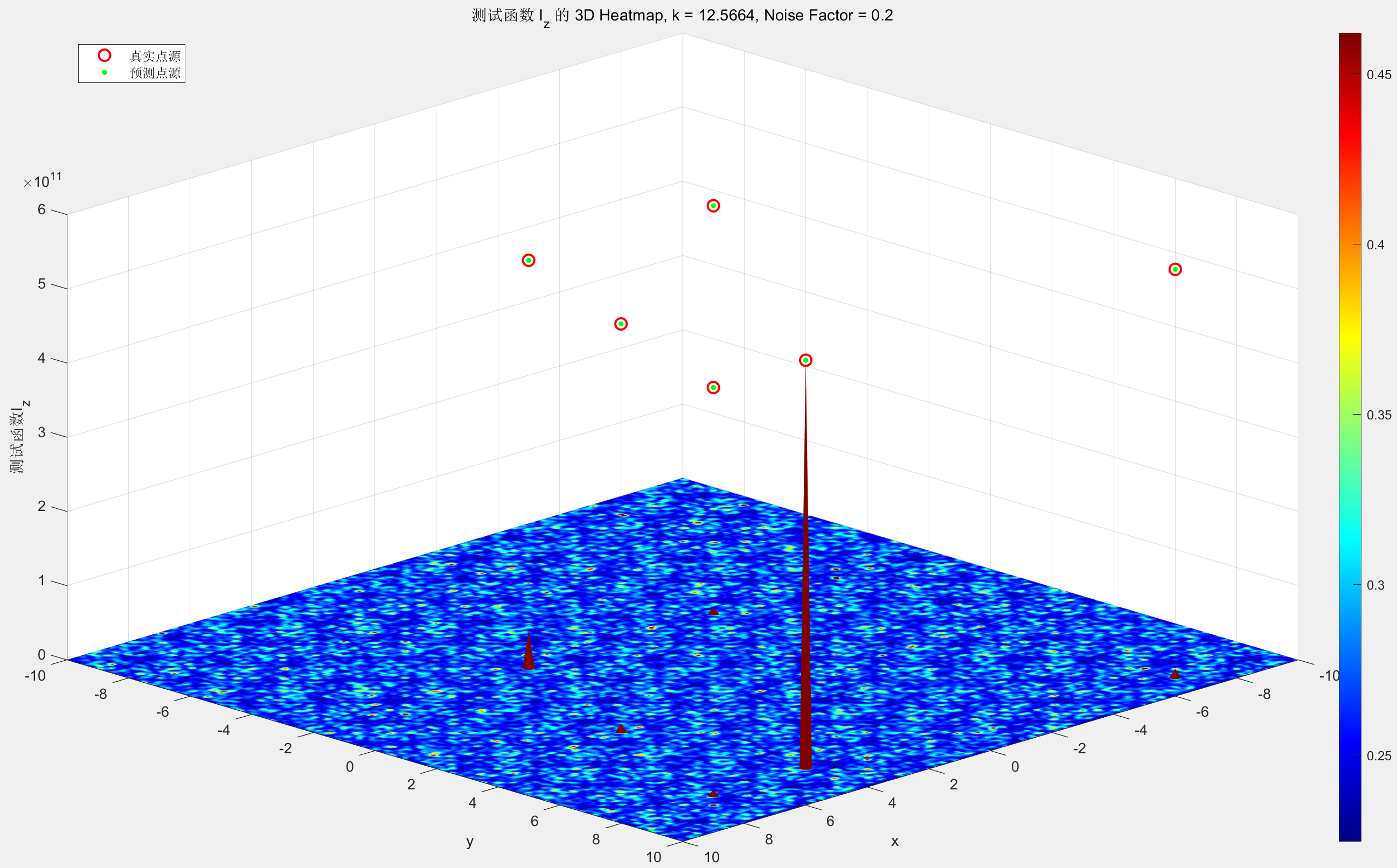}
		\caption{$k =4\pi$}
		\label{fig:wave number 4pi.png}
	\end{subfigure}
	\hfill

	\caption{Using different wave number.}
	\label{fig:wave number}
\end{figure}

From the images, it can be observed that different values of $k$ affect the values of the test function; however, the inversion of the point source locations can still be performed normally

\subsection*{Using different \texorpdfstring{$\alpha$}{alpha}}

To validate the effectiveness of the algorithm, we conducted numerical experiments with different $\alpha$. Let $\alpha_1=[1i; 5i; 7i; 4i; 9i; 10i]$, $\alpha_2=[3+2i; 5+3i; -1+5i; -4+1i; 2+7i; -3+6i]$, $\alpha_3=[1i; -3+5i; 5i; 5+8i; 7i; -6+3i]$. During the experiment, we divided the space with a step size of $0.1$, set observation directions number $N = 20$, used wave number $k=2\pi$ and used $\delta = 0.2$. The positions of these point sources were set as $(3, -2), (5, 3), (-7, 9), (4, 8), (-3, -2), (7, 8).$ 
The generated images are shown in Figure~5.

\begin{figure}[ht]
	\centering	
	\begin{subfigure}[b]{0.32\textwidth}
		\includegraphics[width=\linewidth]{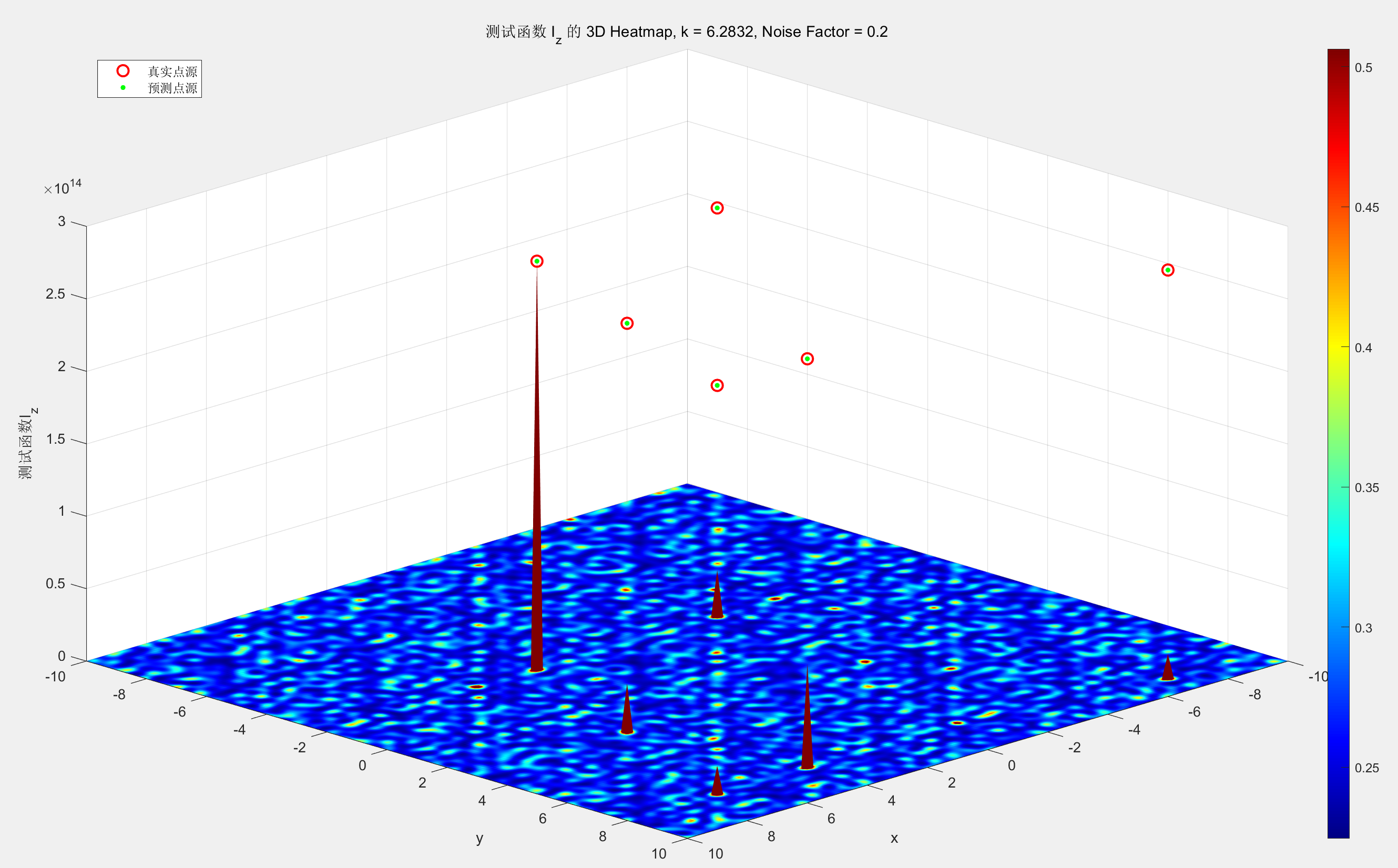}
		\caption{$\alpha_1$}
		\label{fig:alpha_1.png}
	\end{subfigure}
	\begin{subfigure}[b]{0.32\textwidth}
		\includegraphics[width=\linewidth]{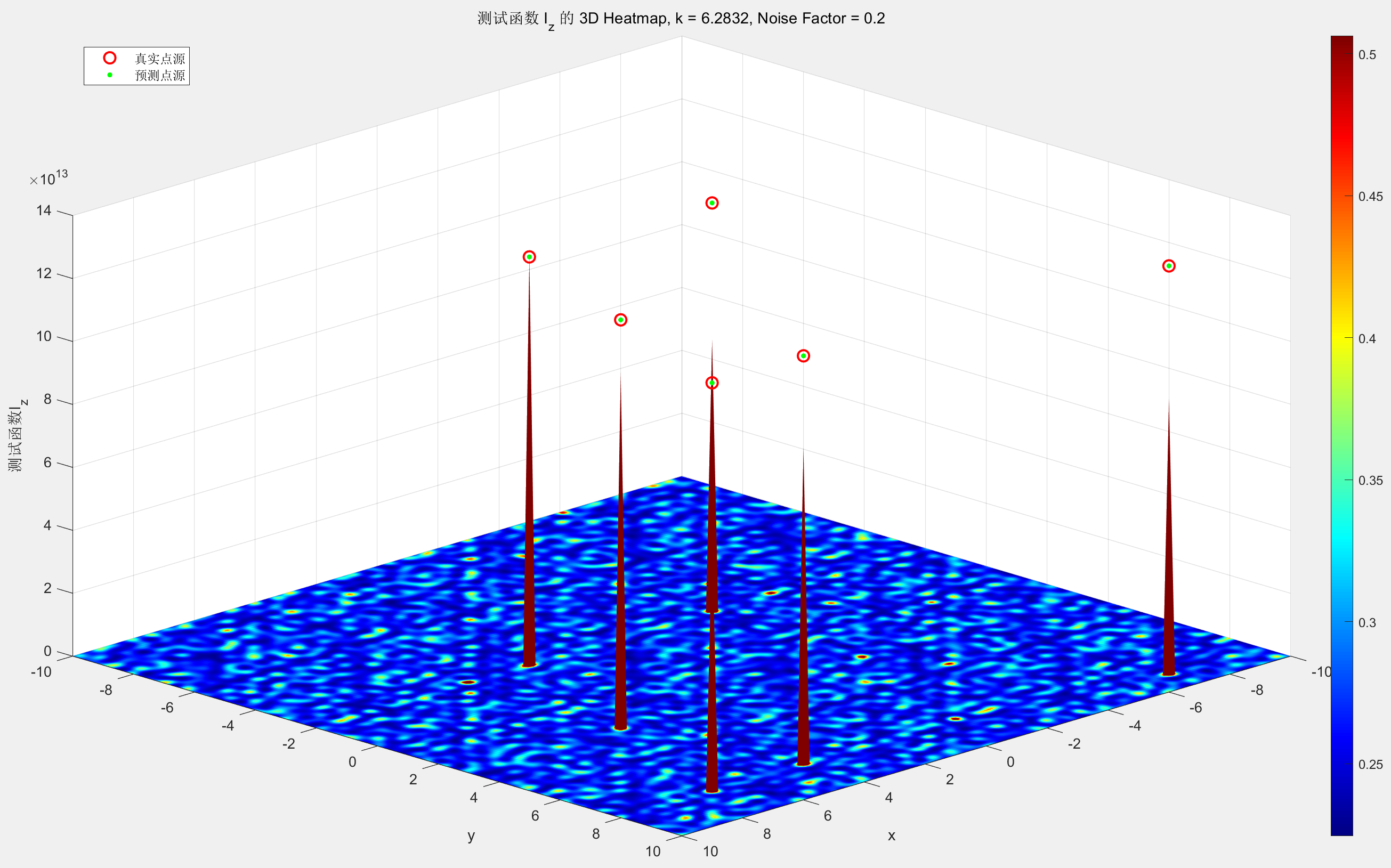}
		\caption{$\alpha_2$}
		\label{fig:alpha_2.png}
	\end{subfigure}
	\hfill
	\begin{subfigure}[b]{0.32\textwidth}
		\includegraphics[width=\linewidth]{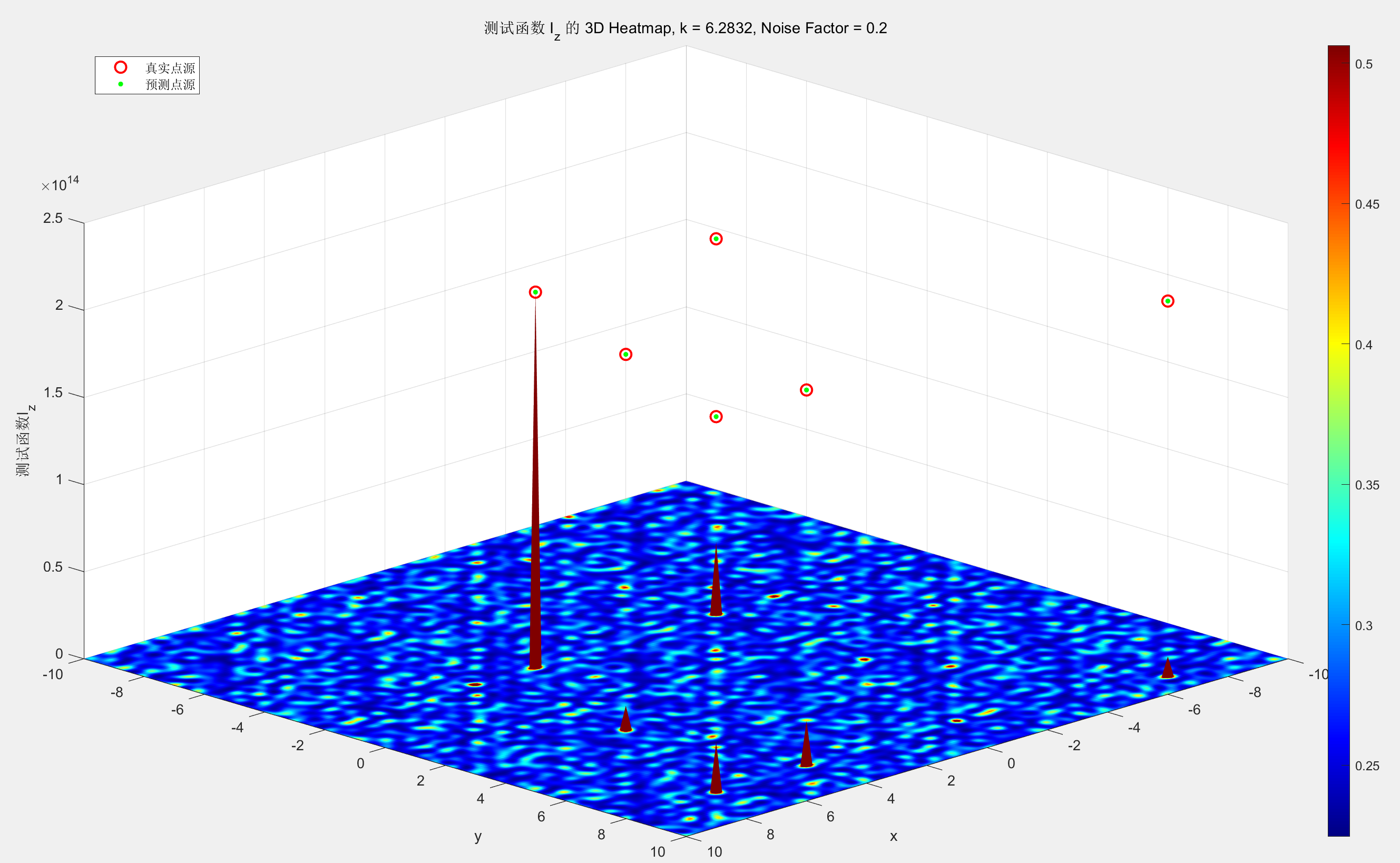}
		\caption{$\alpha_3$}
		\label{fig:alpha_3.png}
	\end{subfigure}
	\hfill
	
	\caption{Using different $\alpha$}
	\label{fig:different_alpha}
\end{figure}

The above numerical examples indicate that, even when $\alpha$ varies, it is still possible to correctly invert and determine the position of the point sources.

\subsection*{Using different \texorpdfstring{$\delta$}{delta}}

To validate the effectiveness of the algorithm, we conducted numerical experiments with different noise. We set $\delta = 0.5, 1, 2$. During the experiment, we divided the space with a step size of $0.1$, set observation directions number $N = 20$, used wave number $k=2\pi$, used $\delta = 0.2$ and set $\alpha=[1i; -3+5i; 5i; 5+8i; 7i; -6+3i]$. The positions of these point sources were set as$ (3, -2), (5, 3), (-7, 9), (4, 8), (-3, -2), (7, 8).$
The generated images are shown in Figure~5.

\begin{figure}[htbp]
	\centering
	\begin{subfigure}[b]{0.32\textwidth}
		\centering
		\includegraphics[width=\textwidth]{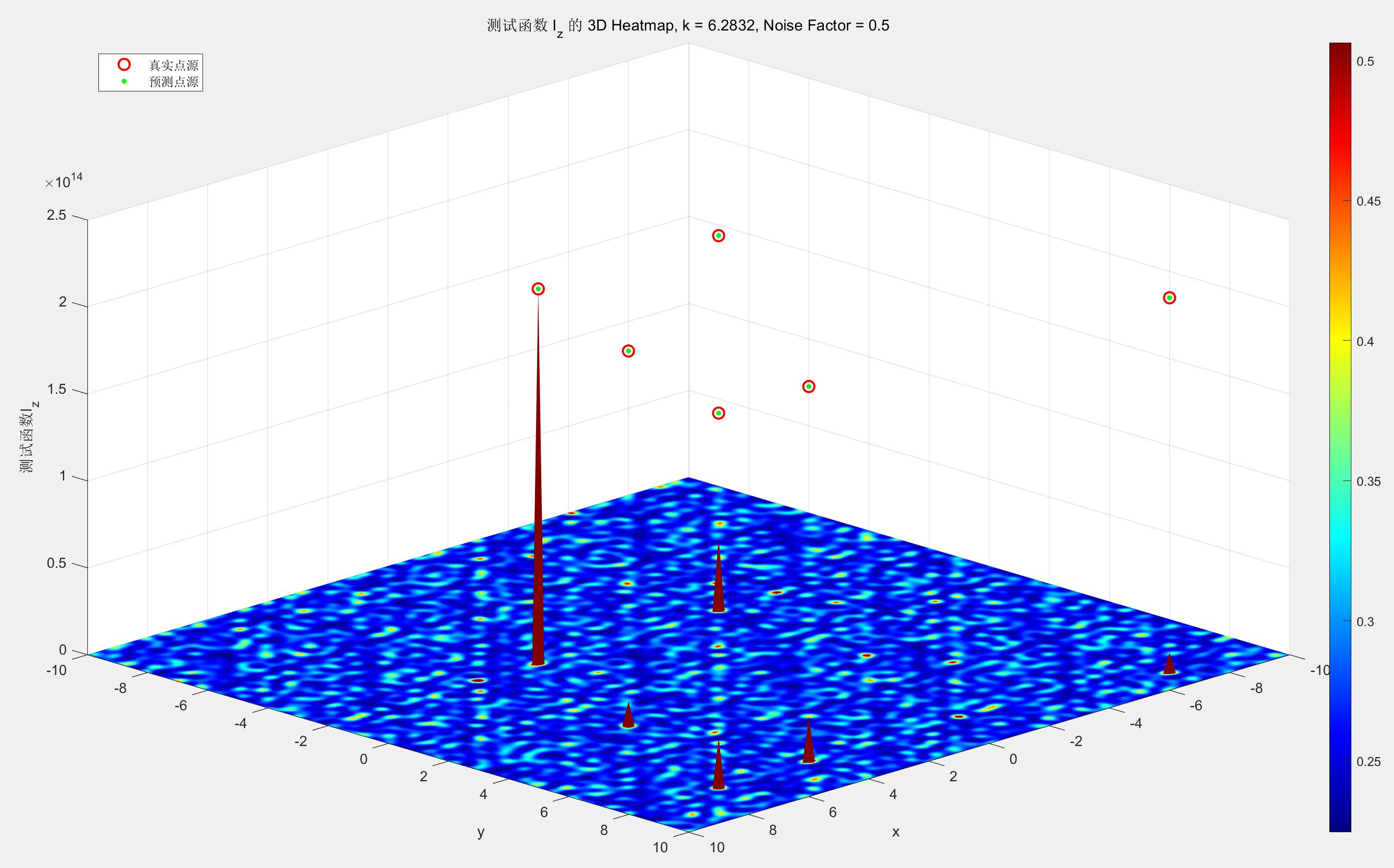}
		\caption{$\delta$ = 0.5}
		\label{fig:noise_factor_1.png}
	\end{subfigure}
	\hfill
	\begin{subfigure}[b]{0.32\textwidth}
		\centering
		\includegraphics[width=\textwidth]{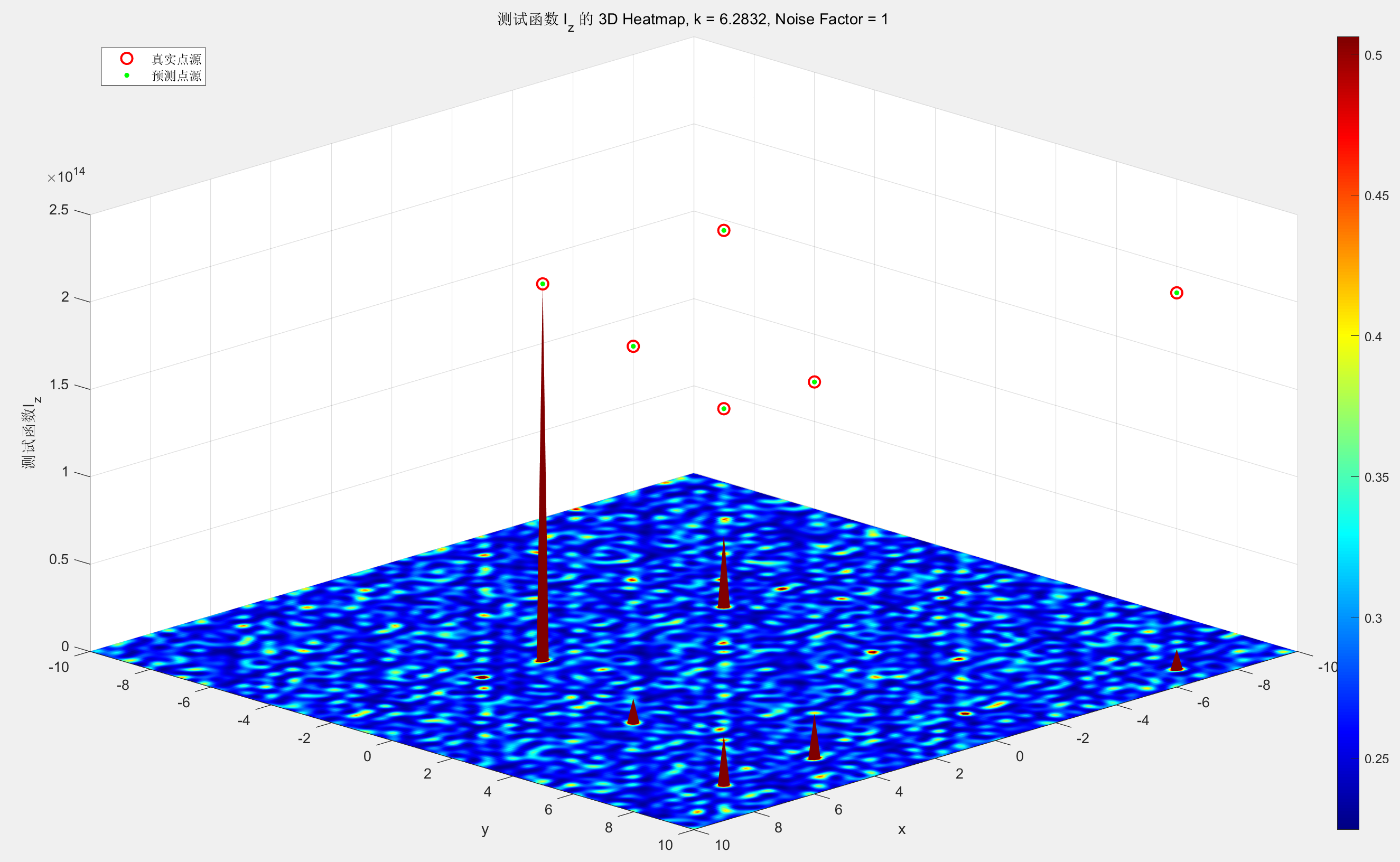}
		\caption{$\delta$ = 1}
		\label{fig:noise_factor_2.png}
	\end{subfigure}
	\hfill
	\begin{subfigure}[b]{0.32\textwidth}
		\centering
		\includegraphics[width=\textwidth]{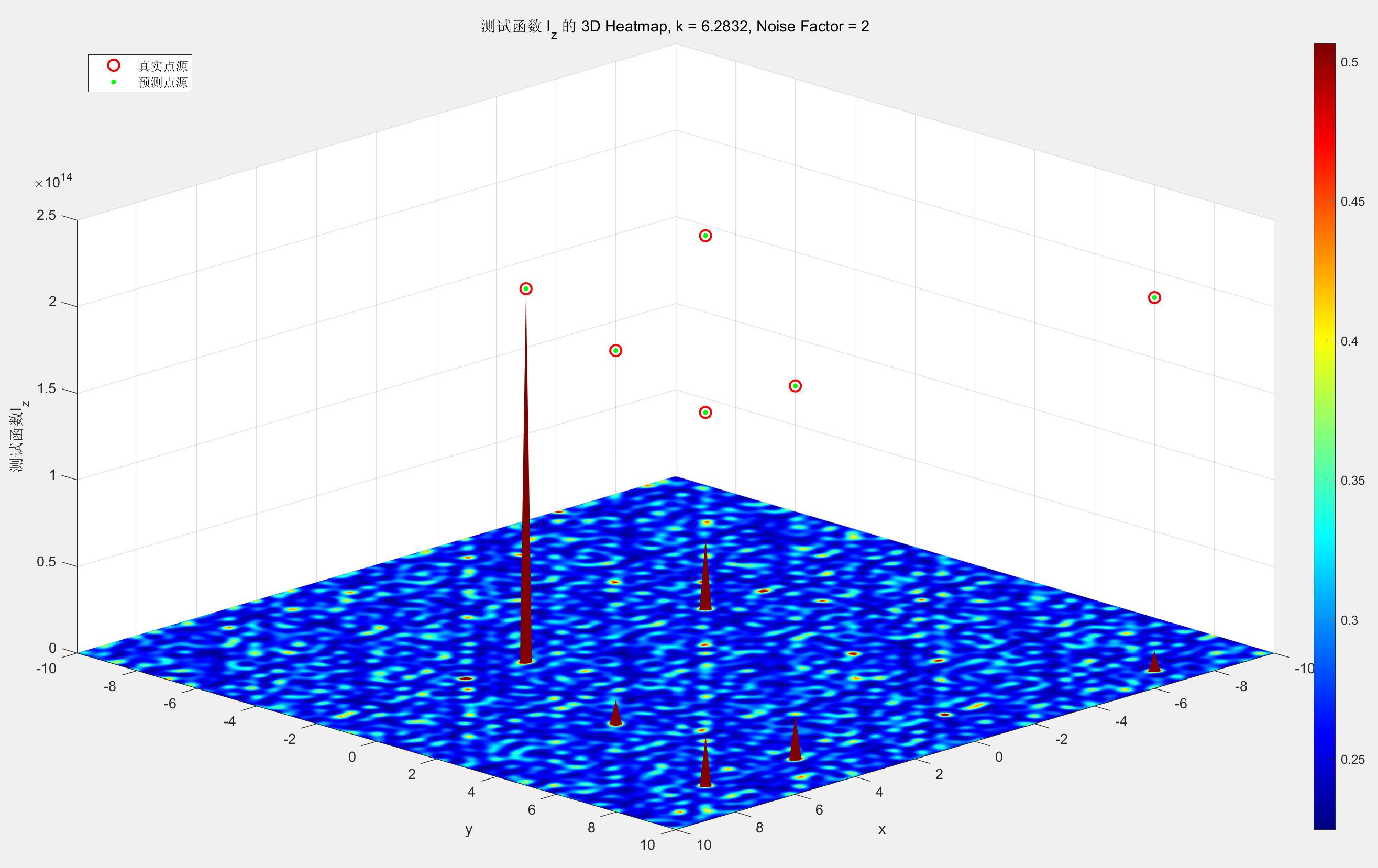}
		\caption{$\delta$ = 2}
		\label{fig:noise_factor_3.png}
	\end{subfigure}
	
	\caption{Using different $\delta$}
	\label{fig:different noisy}
\end{figure}

From the above images, it can be seen that the program designed in this paper is very stable. Even when there is a large amount of external noise, it can still accurately invert and determine the position of the point sources.


\end{document}